\documentclass[a4paper,12pt,twoside,reqno]{smfart}
\usepackage[utf8x]{inputenc}
\usepackage[T1]{fontenc}
\usepackage[frenchb]{babel}
\usepackage{amssymb,amsmath}
\usepackage{smfthm}

\usepackage{stackrel}
\usepackage{pifont}
\usepackage{cases}

\usepackage{graphicx}

\usepackage{enumitem}
\usepackage{multicol}

\usepackage{amscd}
\usepackage[all,cmtip]{xy}

\newenvironment{conj*}{\begin{enonce*}[plain]{Conjecture}}{\end{enonce*}}
\newenvironment{theo*}{\begin{enonce*}[plain]{Th\'eor\`eme}}{\end{enonce*}}

\theoremstyle{plain}
\newcounter{moncompteurtheointro}
\newcounter{moncompteurcorointro}
\newenvironment{theointro}{\refstepcounter{moncompteurtheointro}\begin{enonce*}[plain]{Th\'eor\`eme \themoncompteurtheointro}}{\end{enonce*}}
\newenvironment{corointro}{\refstepcounter{moncompteurcorointro}\begin{enonce*}[plain]{Corollaire \themoncompteurcorointro}}{\end{enonce*}}

\author{Hugues Bauch\`ere}
\address{
Laboratoire de math\'ematiques Nicolas Oresme, CNRS UMR 6139,
Universit\'e de Caen, Campus II, BP 5186,
14032 Caen Cedex, France}
\email{hugues.bauchere@unicaen.fr}
\urladdr{http://www.math.unicaen.fr/~bauchere/}

\title[Minoration de la hauteur pour les modules de Drinfeld CM]
 {Minoration de la hauteur canonique pour les modules de Drinfeld \`a multiplications complexes}

\date{\today}

\DeclareMathOperator{\id}{id}

\DeclareMathOperator{\End}{End}

\DeclareMathOperator{\Gal}{Gal}

\DeclareMathOperator{\sgn}{sgn}

\newcommand{\N}{\mathbb{N}}
\newcommand{\R}{\mathbb{R}}
\newcommand{\C}{\mathbb{C}}
\newcommand{\Q}{\mathbb{Q}}

\newcommand{\F}{\mathbb{F}}
\newcommand{\Z}{\mathbb{Z}}

\newcommand{\Ocal}{\mathcal{O}}

\newcommand{\Kab}{K^{\mathrm{ab}}}

\newcommand{\Zcal}{\mathcal{Z}}
\newcommand{\Rcal}{\mathcal{R}}

\newcommand{\pcar}{\mathfrak{p}}
\newcommand{\qcar}{\mathfrak{q}}
\newcommand{\kbar}{\overline{k}}
\newcommand{\mcar}{\mathfrak{m}}
\newcommand{\ncar}{\mathfrak{n}}
\newcommand{\OF}{\mathcal{O}_F}
\newcommand{\OK}{\mathcal{O}_K}
\newcommand{\OL}{\mathcal{O}_L}
\newcommand{\Oe}{\mathcal{O}_E}
\newcommand{\nbclass}{\mathfrak{h}}
\newcommand{\hw}{h}
\newcommand{\hcan}{\widehat{h}}
\newcommand{\maxR}{d_m}
\newcommand{\maxB}{d_{\OF}}
\newcommand{\maxN}{d_m}
\newcommand{\tors}{\mathrm{tors}}
\newcommand{\Mcal}{\mathcal{M}}
\newcommand{\cf}{\emph{cf.} }
\newcommand{\ie}{\emph{i.e.} }
\newcommand{\dinfprime}{d}
\newcommand{\Ccar}{\mathfrak{C}}

\newcommand{\vinf}{v_{\infty}}
\newcommand{\vinfprime}{v_{\infty'}}

\newcommand{\Cinf}{\mathbb{C}_{\infty}}
\newcommand{\Kbar}{\overline{K}}

\newcommand{\Ftilde}{\widetilde{F}}
\newcommand{\Kinf}{K_{\infty}}
\newcommand{\Finf}{F_{\infty'}}
\newcommand{\FFinf}{\F_{\infty'}}
\newcommand{\Fbar}{\overline{\F}}

\newcommand{\acar}{\mathfrak{a}}
\newcommand{\bcar}{\mathfrak{b}}
\newcommand{\degOF}{\deg_{\OF}}
\newcommand{\degOK}{\deg_{\OK}}
\newcommand{\degOE}{\deg_{\Oe}}
\newcommand{\degOL}{\deg_{\OL}}
\newcommand{\degR}{\deg_{R}}
\newcommand{\degA}{\deg_{A}}
\newcommand{\OcalHFplus}{\Ocal_{H_F^+}}
\newcommand{\HFplus}{H_F^+}

\newcommand{\Hqcar}{H_\qcar}
\newcommand{\OM}{\mathcal{O}_M}
\newcommand{\abar}{\bar{a}}
\newcommand{\bbar}{\bar{b}}
\newcommand{\minibullet}{\scriptscriptstyle{\bullet}}
\newcommand{\cardFtilde}{\varsigma}
\newcommand{\Icar}{\mathfrak{I}}

\begin{document}

\raggedbottom

\frontmatter

\begin{abstract}
 Soient $K$ une extension finie de $\F_q(T)$, $L/K$ une extension galoisienne de groupe de Galois~$G$
 et $E$ le sous-corps de $L$ fix\'e par le centre de $G$.
 On suppose qu'il existe une place finie $v$ de $K$ telle que les degr\'es locaux de $E/K$ au-dessus de $v$ soient born\'es.
 Soit $\phi$ un module de Drinfeld \`a multiplications complexes.
 On donne une minoration effective de la hauteur canonique associ\'ee \`a $\phi$ sur $L$ priv\'e des points de torsion de $\phi$.
 Dans le cadre des corps de nombres, ce probl\`eme a \'et\'e r\'esolu par F. Amoroso, S. David et U. Zannier dans \cite{ADZ}.
\end{abstract}

\alttitle{Lower Bound for the Canonical Height for Drinfeld Modules with Complex Multiplication}

\begin{altabstract}
 Let $K$ be a finite extension of $\F_q(T)$, let $L/K$ be a Galois extension with Galois group~$G$
 and let $E$ be the subfield of $L$ fixed by the center of $G$.
 Assume that there exists a finite place $v$ of $K$ such that the local degrees of $E/K$ above $v$ are bounded.
 Let $\phi$ be a Drinfeld module with complex multiplication.
 We give an effective lower bound for the canonical height of $\phi$ on $L$ outside the torsion points of $\phi$.
 In the number field case, this problem was solved by F. Amoroso, S. David and U. Zannier in \cite{ADZ}.
\end{altabstract}

\altkeywords{Bogomolov, Drinfeld module, canonical height, lower bound}

\maketitle

\renewcommand{\contentsname}{Sommaire}
\tableofcontents

\section*{Introduction}

 Soient $A:=\F_q[T]$ l'anneau des polyn\^omes en l'ind\'etermin\'ee $T$
 et \`a coefficients dans le corps fini~$\F_q$,~$k$ son corps des fractions
 et $\bar{k}$ une cl\^oture alg\'ebrique de $k$.
 
 Soit $\phi$ un $A$-module de Drinfeld d\'efini sur $\kbar$ et de rang $r$.
 Dans \cite{DENIS1}, L. Denis d\'efinit la hauteur canonique $\hcan_\phi:\kbar\longrightarrow\R_+$ associ\'ee \`a un $A$-module de Drinfeld $\phi$.
 Il montre que cette hauteur s'annule uniquement sur les points de torsion de~$\phi$.

 Soit $L/k$ une extension de $k$. Par analogie avec la terminologie utilis\'ee dans \cite{BOMZAN}, on dira que $L$ a la propri\'et\'e (B,$\phi$)
 s'il existe une constante strictement positive qui minore $\hcan_\phi$ sur $L$
 priv\'e des points de torsion de~$\phi$.

 S. David et A. Pacheco ont montr\'e dans \cite{DAVPAC} que pour tout module de Drinfeld $\phi$,
 si $K/k$ est une extension finie, alors
 la cl\^oture ab\'elienne de $K$ avait la propri\'et\'e (B,$\phi$).
 Dans cet article nous g\'en\'eralisons, dans le cadre des modules de Drinfeld \`a multiplications complexes,
 ce r\'esultat (on renvoie au \S\ref{sec:prelim} pour les d\'efinitions pr\'ecises des notions utilis\'ees).


\begin{theointro}\label{thlehmerB}\label{THLEHMERB}
 Soit $\phi$ un $A$-module de Drinfeld d\'efini sur $\kbar$ et \`a multiplications complexes.
 Soient $K/k$ une extension finie et $L/K$ une extension galoisienne (finie ou infinie) de groupe de Galois $G$.
 Soient $H$ un sous-groupe du centre de $G$ et $E\subseteq L$ le sous-corps fix\'e par $H$.
 Soit $d_0>0$ un entier. On suppose qu'il existe une place finie $v$ de $K$
 telle que pour toute place $w$ de~$E$ au-dessus de $v$, on ait $\left[E_w:K_v\right]\leq d_0$.
 Alors~$L$ a la propri\'et\'e (B,$\phi$).
 Plus pr\'ecis\'ement,
 il existe une constante $c_0>0$ qui ne d\'epend que de $\phi$ telle que pour tout $\alpha\in L$ non de torsion pour $\phi$, on~ait:
 \[\hcan_\phi(\alpha)\geq q^{-c_0\,\deg v\,d_0^2\,[K:k]}\text{.}\]
\end{theointro}
\indent On illustre la tour d'extensions que nous venons de d\'ecrire par le diagramme suivant:
\[\xymatrix@-4ex{
  L \\
  E \ar@{-}[u]_{<\Zcal(G)} \\
  K \ar@{-}[u] \ar@/^3pc/@{--}[uu]^{G} \\
  k \ar@{-}[u]
 }\]

Notons que lorsque l'extension $L/K$ est finie, alors $L$ a la propri\'et\'e (B,$\phi$) d'apr\`es le th\'eor\`eme de Northcott
(\cf th\'eor\`eme \ref{hautcan}, point $4$).

Le th\'eor\`eme \ref{thlehmerB} est donc surtout int\'eressant lorsque l'extension galoisienne $L/K$ est infinie.

Remarquons que ce th\'eor\`eme est l'analogue pour les modules de Drinfeld \`a multiplications complexes
d'un r\'esultat obtenu par F. Amoroso, S. David et U. Zannier dans \cite{ADZ}
dans le cas du groupe multiplicatif.
Dans ce cadre, m\^eme pour une extension ab\'elienne $L$ d'un corps de nombres~$K$,
le fait d'obtenir une minoration d\'ependant seulement du degr\'e $[K:\Q]$ n'\'etait pas possible \`a l'aide des techniques de \cite{AMOZAN1}.
Pour ce faire les auteurs ont d\^u modifier leur m\'ethode (\cf \cite{AMOZAN2}), notamment pour ne pas utiliser le th\'eor\`eme de Kronecker-Weber.
Pour ce qui est des modules de Drinfeld, la minoration obtenue dans \cite{DAVPAC} n'est pas explicite.
Dans le cadre de modules CM, leur m\'ethode permettrait probablement d'obtenir une minoration d\'ependant seulement du degr\'e $[K:k]$.
Ici nous utilisons, comme dans \cite{ADZ}, la m\'ethode de \cite{AMOZAN2}, ce qui nous permet d'obtenir en particulier,
dans le cadre des modules CM, une version du th\'eor\`eme de \cite{DAVPAC} avec une constante d\'ependant seulement du degr\'e $[K:k]$:

\begin{corointro}
 Soit $\phi$ un $A$-module de Drinfeld d\'efini sur $\kbar$ \`a multiplications complexes.
 Soient $K/k$ une extension finie et $\Kab$ la cl\^oture ab\'elienne de $K$.
 Alors $\Kab$ a la propri\'et\'e~(B,$\phi$).
 Plus pr\'ecis\'ement,
 il existe une constante $c_0>0$ qui ne d\'epend que de $\phi$ telle que pour tout $\alpha\in\Kab$ non de torsion pour $\phi$, on ait:
 \[\hcan_\phi(\alpha)\geq q^{-c_0\,[K:k]^2}\text{.}\]
\end{corointro}

\begin{proof}
 On applique le th\'eor\`eme \ref{thlehmerB} dans le cas $L=\Kab$ en prenant $H=G$. Alors ${E=K}$,
 on peut ainsi prendre $d_0=1$ et choisir n'importe quelle place finie et non triviale de $K$. 
 Prenons une place finie~$v$ de $K$ telle que $\deg v\leq[K:k]$
 (une telle place existe, il suffit par exemple de prendre une place de $K$ au-dessus de la place en $T$ de $k$).
 D'apr\`es le th\'eor\`eme \ref{thlehmerB}, il existe une constante $c_0>0$ ne d\'ependant que de $\phi$
 telle que pour tout $\alpha\in\Kab$ non de torsion pour $\phi$, on a:
 \[\hcan_\phi(\alpha)\geq q^{-c_0\,[K:k]^2}\text{,}\]
 ce qui permet de conclure.
\end{proof}

S'il existe une place finie $v$ de $K$
telle que pour toute place $w$ de~$L$ au-dessus de $v$, on ait ${\left[L_w:K_v\right]\leq d_0}$ pour un certain entier $d_0$,
alors, en consid\'erant le sous-groupe trivial 
du centre du groupe de Galois de l'extension $L/K$, on obtient le corollaire suivant:

\begin{corointro}
 Soit $\phi$ un $A$-module de Drinfeld d\'efini sur $\kbar$ et \`a multiplications complexes.
 Soient $K/k$ une extension finie et $L/K$ une extension galoisienne de groupe de Galois $G$.
 Soit~$d_0>0$ un entier. S'il existe une place finie $v$ de $K$
 telle que pour toute place $w$ de~$L$ au-dessus de $v$, on ait $\left[L_w:K_v\right]\leq d_0$,
 alors $L$ a la propri\'et\'e (B,$\phi$).
 Plus pr\'ecis\'ement,
 il existe une constante $c_0>0$ qui ne d\'epend que de $\phi$ telle que pour tout $\alpha\in L$ non de torsion pour $\phi$, on ait:
 \[\hcan_\phi(\alpha)\geq q^{-c_0\,\deg v\,d_0^2\,[K:k]}\text{.}\]
\end{corointro}

Afin de d\'emontrer le th\'eor\`eme \ref{thlehmerB}, nous aurons besoin de fa\c{c}on cruciale d'une congruence
(un \og rel\`evement du Frobenius \fg, \cf proposition \ref{propCongModDNormalise} ci-apr\`es).
Celle-ci n'\'etant disponible que pour les modules de Drinfeld de signe normalis\'e (\cf \S\ref{subsectionThHayes} pour la d\'efinition),
nous nous ram\`enerons au cas de modules de Drinfeld de signe normalis\'e.
Nous nous ram\`enerons \'egalement au cas d'extensions finies.
Plus pr\'ecis\'ement, nous d\'eduirons le th\'eor\`eme \ref{thlehmerB} du th\'eor\`eme suivant:

\begin{theointro}\label{thlehmerBpsifini}\label{THLEHMERBPSIFINI}
 Soient $F/k$ une extension CM finie, $\OF$ son anneau d'entiers,
 $d$ le degr\'e sur $\F_q$ de l'unique place de $F$ au-dessus de $\infty=\frac{1}{T}$
 et~$\psi$ un $\OF$-module de Drinfeld de signe normalis\'e.
 Soient $K/k$ une extension finie telle que~$\psi$ soit d\'efini sur $K$
 et $L/K$ une extension galoisienne finie de groupe de Galois $G$.
 Soient $H$ un sous-groupe du centre de $G$ et $E\subseteq L$ le sous-corps fix\'e par~$H$.
 Soient $d_1>0$ un entier et $v$ une place finie de $K$.
 On suppose que pour toute place $w$ de $E$ au-dessus de~$v$, on ait $\left[E_w:K_v\right]\leq d_1$.
 Alors il existe deux constantes $c_1>0$ et $c_2\geq1$ qui ne d\'ependent que de $\psi$
 telles que pour tout $\alpha\in L$ non de torsion pour $\psi$, on~ait:
 \[\hcan_{\psi}(\alpha)\geq\frac{q^{-c_1\,\deg v\,d_1^2\,[K:k]}}{c_2\,d_1^{q^d-1}\,[K:k]^{q^r}}\text{.}\]
\end{theointro}

Le plan de l'article est le suivant.
Au \S\ref{sec:prelim}, nous rassemblons les pr\'eliminaires n\'ecessaires sur les corps de fonctions en caract\'eristique positive,
sur les modules de Drinfeld et sur leur hauteur canonique.
Au \S\ref{chapB}, on d\'emontre le th\'eor\`eme \ref{thlehmerBpsifini}.
Enfin, au \S\ref{demothlehmerB}, on en d\'eduit le th\'eor\`eme \ref{thlehmerB}.

\mainmatter

\section{Pr\'eliminaires}\label{sec:prelim}

\subsection{Notations g\'en\'erales}\label{subsectionnotagen}
\leavevmode\par

 Soit $p$ un nombre premier et $q$ une puissance de $p$. On consid\`ere $\F_q$ le corps fini \`a $q$ \'el\'ements,
 $A:=\F_q[T]$ l'anneau des polyn\^omes \`a coefficients dans $\F_q$ en l'ind\'etermin\'ee $T$ et $k$ le corps des fractions de $A$.
 Soit $a\in A$, on note $\deg_T a$ le degr\'e de $a$ en tant que polyn\^ome en $T$.
 
 Soient $\infty$ la place en $\frac{1}{T}$ et
 $k_{\infty}:=\F_q\!\left(\!\left(\frac{1}{T}\right)\!\right)$ le corps des s\'eries de Laurent sur $\F_q$,
 c'est un compl\'et\'e de $k$ pour la valuation $\frac{1}{T}$-adique $v_{\infty}=-\deg_T$.
 On fixe $\bar{k}_{\infty}$ une cl\^oture alg\'ebrique de $k_{\infty}$
 et $\C_{\infty}$ un compl\'et\'e de $\bar{k}_{\infty}$ pour la valuation $v_{\infty}$ que l'on a \'etendue naturellement \`a $\bar{k}_{\infty}$.
 Ensuite, on normalise sur~$\C_{\infty}$ la valeur absolue associ\'ee \`a la valuation $v_{\infty}$ en posant $|\cdot|_{\infty}:=q^{-v_{\infty}(\cdot)}$.
 On note respectivement~$\bar{k}$ la fermeture alg\'ebrique de $k$ dans $\C_{\infty}$
 et $\Fbar_q$ celle de $\F_q$ dans $\Cinf$.
 
 Soit $K/k$ une extension finie, on d\'esigne par $[K:k]$ le degr\'e de cette extension,
 par $\OK$ la fermeture int\'egrale de $A$ dans $K$,
 par $\Kbar$ la fermeture alg\'ebrique de~$K$ dans $\Cinf$
 et par $\Kinf$ le compl\'et\'e de $K$ dans $\Cinf$ pour la valuation $\vinf$.

 Si $L$ est une extension galoisienne de $K$, on note $\Gal(L/K)$ son groupe de Galois.

 Si $E$ est un ensemble, on note $\#E$ son cardinal.


\subsection{Places et degr\'es}
\leavevmode\par

 Soient $K$ une extension finie de $k$ et $\mathcal{M}_K$ l'ensemble des places non triviales de $K$.
 Si $v\in\Mcal_K$, on note $v:K\mapsto\Z\cup\{\infty\}$ la valuation associ\'ee
 et $|\cdot|_v:=q^{-v(\cdot)}$ la valeur absolue normalis\'ee associ\'ee.
 On note \'egalement $\mathcal{O}_v:=\big\{x\in K\,\big|\,v(x)\geq0\big\}$\label{notaOv} son anneau de valuation,
 ${\mcar_v:=\big\{x\in K\,\big|\,v(x)>0\big\}}$ son id\'eal maximal,
 $\F_v:=\mathcal{O}_v/\mcar_v$ son corps r\'esiduel
 et ${\deg v:=[\F_v:\F_q]}$\label{notadegv} son degr\'e sur $\F_q$.

 Soit $L$ une extension finie de $K$. Si $w$ est une place de $L$ au-dessus de~$v$ (ce que l'on note $w|v$),
 on pose $f(w/v):=[\F_w:\F_v]$ le degr\'e r\'esiduel de $w$ sur $v$
 et $e(w/v)$ l'indice de ramification de~$w$ sur~$v$,
 c'est-\`a-dire l'unique entier tel que pour tout $\alpha\in K$ on a:
 \[w(\alpha)=e(w/v)\,v(\alpha)\text{.}\]

 On note $K_v$ le compl\'et\'e de $K$ en la place $v$ et $L_w$ celui de $L$ en $w$.
 On appelle degr\'e local de $w$ sur~$v$ le degr\'e de l'extension $L_w/K_v$.


 Soient $F/k$ une extension finie
 et $R$ un \emph{ordre} de $F$, \ie un sous-anneau de $\OF$ ayant $F$ comme corps des fractions.
 Rappelons que $\OF$ est l'\emph{ordre maximal} de~$F$ (pour l'inclusion).

 Pour tout id\'eal non nul $\acar$ de $R$, on d\'efinit:
 \[
  \deg_R\acar:=\log_q\#\!\left(R/\acar\right)\text{,}
 \]
 o\`u $\log_q$ est le logarithme de base $q$. \\
 On \'etend la d\'efinition pr\'ec\'edente \`a tout \'el\'ement non nul $\gamma\in R$ en posant:
 \[
  \deg_R\gamma:=\deg_R(\gamma)\text{,}
 \]
 o\`u $(\gamma):=\gamma R$ est l'id\'eal de $R$ engendr\'e par $\gamma$.
 En particulier, pour tout $a\in A\setminus\{0\}$, on a:
 \[
  \deg_A a=\deg_T a\text{.}
 \]
 On notera \'egalement que si $\pcar$ est un id\'eal premier non-nul de $\OK$ et si~$v_{\pcar}$ est la valuation associ\'ee \`a $\pcar$,
 alors on a $\deg_{\OK}\pcar=\deg v_{\pcar}$.


\subsection{Modules de Drinfeld}\label{subsectionModDrinfeld}
\leavevmode\par

 Soient $L\subset\Cinf$ une extension de $k$ et $\tau$ l'application de Frobenius associ\'ee \`a $\F_q$ (\emph{i.e.} $\tau(x)=x^q$ pour tout $x\in\Cinf$),
 on note $L\{\tau\}$ l'anneau des polyn\^omes de \"Ore sur $L$. Si $P\in L\{\tau\}$,
 on note $\deg_\tau P$ le degr\'e de $P$ en tant que polyn\^ome en $\tau$ et $D P$ le coefficient du mon\^ome en $\tau^0$ de $P$.

%
%
 
 On peut g\'en\'eraliser la th\'eorie classique des modules de Drinfeld
 aux ordres d'extensions CM de $k$ en suivant \cite{HAYES}.
 Soit $F/k$ une extension finie
 de type CM,
 c'est-\`a-dire que $F$ ne poss\`ede qu'une seule place au-dessus de~$\infty$.
 Soit $R$ un ordre de $F$.
 On dit qu'un morphisme de $\F_q$-alg\`ebres $\rho:R\longrightarrow\Cinf\{\tau\}$ est un
 \emph{$R$-module de Drinfeld}
 si pour tout $a\in R$,
 on a $D\rho_a=a$ et s'il existe $a\in R$ v\'erifiant $\rho_a\neq a\,\tau^0$.

 D'apr\`es la proposition $2.3$ de \cite{HAYES}, il existe un entier $r>0$ tel que pour tout $a\in R\setminus\{0\}$,
 on ait: $\deg_\tau \rho_a=r\,\deg_R a\text{.}$
 Cet entier est appel\'e le \emph{rang} de $\rho$\label{notarang}.

 Soient $\rho$ et $\varrho$ deux $R$-modules de Drinfeld.
 Un \emph{morphisme}
 de $\rho$ vers~$\varrho$ est un polyn\^ome $P\in\Cinf\{\tau\}$ v\'erifiant $P\rho_a=\varrho_a P$ pour tout $a\in R$.
 Un morphisme non nul de $\rho$ vers $\varrho$ est appel\'e \emph{isog\'enie}.
 On remarque que si~$\rho$ et $\varrho$ sont isog\`enes, alors ils ont n\'ecessairement m\^eme rang.
 Un \emph{isomorphisme}
 de $\rho$ vers $\varrho$ est une isog\'enie inversible, c'est donc un \'el\'ement de~$\Cinf^*$.

 On note $H_\rho\subset\Cinf$ le corps engendr\'e sur $k$
 par les coefficients des polyn\^omes $\left\{\rho_a\right\}_{a\in R}$ de~$\Cinf\{\tau\}$.
 On dira que $\rho$ est d\'efini
 sur $L\subset\Cinf$ si $H_\rho\subset L$, \ie $\rho:R\longrightarrow L\{\tau\}$.

 Notons $t:=[F:k]$.
 Si~$(e_1,\dots,e_t)$ est une base de $R$ en tant que $A$-module, alors $H_\rho$ est le corps engendr\'e sur $k$
 par les coefficients des polyn\^omes $\rho_T\in\Cinf\{\tau\}$ et $\left\{\rho_{e_i}\right\}_{1\leq i\leq t}$.
 On remarque par ailleurs que si $\rho$ est d\'efini sur $\kbar$,
 alors $H_\rho$ est une extension finie de $k$ contenant~$F$ (car $D\rho_a=a$, pour tout $a\in R$).
 


%


\subsection{Endomorphismes d'un module de Drinfeld}\label{subsectionEndoModDrinfeld}
\leavevmode\par

 Dor\'enavant par module de Drinfeld on entendra module de Drinfeld d\'efini sur $\kbar$.
 
 Soit $\phi$ un $A$-module de Drinfeld de rang $r$.
%
 On note $\End(\phi)\subset\kbar\{\tau\}$ l'ensemble des endomorphismes de $\phi$ sur $\kbar$.
 L'anneau $\End(\phi)$ est un $A$-module de rang $\leq r$ (\cf cor. de la prop. $2.4.$ de \cite{DRINFELDI}) via l'action de $\phi$.
 De plus, $\End(\phi)$ se plonge naturellement dans $\kbar$ via l'application $P\longmapsto D P$.
 Si on note~$R$ l'image de $\End(\phi)$ dans $\kbar$, alors d'apr\`es le corollaire $4.7.15$ de \cite{GOSS},
 $R$ est un ordre d'une extension finie~$F$ de~$k$.
 De plus cette extension est de type CM (voir la preuve du th\'eor\`eme $4.7.17$ de \cite{GOSS})
 de degr\'e \'egal au rang $t$ de~$R$ en tant que $A$-module.
 L'isomorphisme de $A$-modules: 
 \begin{equation}\label{eqplongementEndphikbar}
   \phi':R\stackrel{\sim}{\longrightarrow}\End(\phi)
 \end{equation}
 d\'efinit naturellement une structure de $R$-module de Drinfeld de rang $\frac{r}{t}$
 qui \'etend celle de $\phi$, \ie:
 \begin{equation}\label{eqextscalairesModDrinfeld}
  \forall a\in A, \phi'_a=\phi_a\text{.}
 \end{equation}

 On dit que $\phi$ est \`a
 \emph{multiplications complexes}
 (en abr\'eg\'e~CM) si $\phi'$ est de rang $1$, autrement dit si $\End(\phi)$ est un $A$-module de rang $r$ (car alors~${t=r}$).


\subsection{Points de torsion et id\'eaux}\label{subsectionPtTorsId}
\leavevmode\par

 Soient $F/k$ une extension finie de type CM, $R$ un ordre de $F$,
 $\phi'$\label{notaphiprime} un $R$-module de Drinfeld de rang~$r'$
 et $H_{\phi'}$ le corps des coefficients de $\phi'$ (\cf~\S\ref{subsectionModDrinfeld}).

 Soit $a\in R\setminus\{0\}$. Un \emph{point de $a$-torsion}
 de $\phi'$ est un \'el\'ement $\delta\in\kbar$ tel que $\phi'_a(\delta)=0$.
 On note~$\phi'[a]$\label{notaphiprimea} l'ensemble des points de $a$-torsion.

 De fa\c{c}on g\'en\'erale, si $\acar$ est un id\'eal non nul de $R$,
 un \emph{point de $\acar$-torsion}
 de $\phi'$ est un \'el\'ement $\delta\in\kbar$ tel que $\phi'_a(\delta)=0$ pour tout $a\in\acar$.
 On note~$\phi'[\acar]$\label{notaphiprimeacar}
 l'ensemble des points de $\acar$-torsion, c'est un $R$-module via $\phi'$.
 On note \'egalement $\phi'\!\left[\acar^\infty\right]:=\cup_{i\geq0}\phi'\!\left[\acar^i\right]$.

 Si $L\subseteq\kbar$ est une extension de $k$, on note $\phi'(L)_\tors$\label{notaphiprimeLtors}
 l'ensemble des points de torsion de $\phi'$ dans $L$.

%
%

Soit $\delta\in\kbar$ un point de torsion de $\phi'$.
On appelle \emph{ordre}
de $\delta$ l'id\'eal
$\mathfrak{I}_{\phi'}(\delta):=\left\{a\in R\,\big|\,\phi'_a(\delta)=0\right\}$\label{notaIdelta},
c'est le plus grand id\'eal annulateur de $\delta$ via l'action de $\phi'$.
On remarque imm\'ediatement que $\delta$ est d'ordre premier \`a $\acar$
si et seulement s'il existe un id\'eal non nul $\bcar$ de $R$ premier \`a $\acar$ tel que $\delta\in\phi'[\bcar]$
(en effet, dans ce cas $\bcar\subseteq\Icar_{\phi'}(\delta)$).

Lorsque $R=\OF$, le lemme suivant montre qu'on peut d\'ecomposer un point de torsion suivant un id\'eal premier non nul fix\'e de $\OF$.

\begin{lemm}\label{lemsompttorsion}
 Soient $\mcar$ un id\'eal premier non nul de $\OF$ et $\delta\in\kbar$ un point de torsion pour~$\phi'$.
 Alors, il existe $\delta_1\in H_{\phi'}(\delta)$ un point de torsion pour $\phi'$ d'ordre une puissance de $\mcar$
 et $\delta_2\in H_{\phi'}(\delta)$ un point de torsion pour $\phi'$ d'ordre premier \`a $\mcar$ tels que $\delta=\delta_1+\delta_2$.
\end{lemm}

\begin{proof}
 Soit $a\!\in\!\OF\!\setminus\{0\}$ tel que $\phi'_a(\delta)\!=\!0$.
 \'Ecrivons ${(a)=\mcar^g\,\bcar}$ avec $g\geq0$ un entier et~$\bcar$ un id\'eal de $\OF$ premier \`a $\mcar$.
 Alors $\mcar^g$ et $\bcar$ sont premiers entre eux, donc $\mcar^g+\bcar=\OF$. Ainsi, il existe $u\in\bcar$ et $v\in\mcar^g$ tels que $u+v=1$.
 D'o\`u
 \[\delta=\phi'_1(\delta)=\phi'_u(\delta)+\phi'_v(\delta)\text{.}\]
 Posons $\delta_1:=\phi'_u(\delta)$ et $\delta_2:=\phi'_v(\delta)$.
 On va v\'erifier que $\delta_1\in\phi'[\mcar^\infty]$ et $\delta_2\in\phi'[\bcar]$.
 Soit $m$ un \'el\'ement de~$\mcar^g$. On a $m\,u\in\mcar^g\,\bcar=(a)$, il existe donc $c\in\OF$ tel que $m\,u=a\,c$.
 On en d\'eduit
 \[\phi'_m(\delta_1)=\phi'_{m\,u}(\delta)=\phi'_{a\,c}(\delta)=\phi'_c\left(\phi'_a(\delta)\right)=0\text{.}\]
 Ceci \'etant vrai pour tout $m\in\mcar^g$, on a bien $\delta_1\in\phi'[\mcar^g]$ comme annonc\'e.
 Un raisonnement analogue montre que $\delta_2\in\phi'[\bcar]$. Le lemme est donc d\'emontr\'e.
\end{proof}



\subsection{Th\'eorie de Hayes}\label{subsectionThHayes}
\leavevmode\par

 Nous rassemblons dans ce paragraphe  des r\'esultats dus \`a Hayes qui seront essentiels pour la suite
 et nous d\'emontrons une congruence qui est un point clef de la preuve du th\'eor\`eme principal.

 Soient $F/k$ une extension finie de type CM et $R$ un ordre de $F$.
 On note $\infty'$ l'unique place de $F$ au-dessus de $\infty$ et $d:=\deg\infty'$\label{notadinfprime} son degr\'e.
 Soient~$\FFinf$\label{notaFFinf} le corps r\'esiduel de $\infty'$ (qui est donc isomorphe \`a $\F_{q^d}$),
 $\vinfprime$\label{notavinfprime} l'unique extension de $\vinf$ \`a $F$ et $\Finf$ le compl\'et\'e de~$F$ pour~$\vinfprime$.

 Soit $\rho$ un $R$-module de Drinfeld et $\acar$ un id\'eal non nul de $R$.
 On note $\rho_\acar$ le g\'en\'erateur unitaire de l'id\'eal \`a gauche de $\kbar\{\tau\}$ engendr\'e par la famille $\left\{\rho_a\right\}_{a\in\acar}$.
 D'apr\`es le \S$3$ de \cite{HAYES},
 il existe alors un unique $R$-module de Drinfeld $\varrho$ tel que pour tout $a\in A$, on ait:
 \[\rho_\acar\,\rho_a=\varrho_a\,\rho_\acar\text{.}\]
 On notera $\acar*\rho$\label{notaacaretoilerho} le $R$-module $\varrho$, en suivant \cite{HAYES}.


 De plus, si on note $\End(\acar)$\label{notaEndacar}
 l'ensemble des \'el\'ements $a\in\OF$ tels que $a\acar\subseteq\acar$,
 alors, d'apr\`es la proposition $3.2$ de \cite{HAYES}, $\varrho$ est en fait un $\End(\acar)$-module de Drinfeld.

 Nous utiliserons exclusivement le cas particulier suivant pour \'etendre \`a l'ordre maximal l'anneau des scalaires de nos modules de Drinfeld CM.
 Soit $\Ccar:=\left\{a\in\OF\,\big|\,a\OF\subset R\right\}$\label{notaCcar}
 le \emph{conducteur} de $R$ (\ie le plus grand id\'eal de $\OF$ contenu dans $R$).
 On a alors $\End(\Ccar)=\OF$ et d'apr\`es ce qui pr\'ec\`ede, on obtient:

\begin{prop}\label{propmodetenduconducteur}
 Soit $\rho$ un $R$-module de Drinfeld. Notons $\Ccar$ le con\-ducteur de $R$ et $\varphi=\Ccar*\rho$.
 Alors $\rho$ est isog\`ene \`a $\varphi$ via $\rho_\Ccar$
 et on a $\End(\varphi)=\OF$.
\end{prop}

 Soit $\pi\in F$\label{notapi} une uniformisante de $\vinfprime$ (\ie $\vinfprime(\pi)=1$).
 On a alors $\Finf=\FFinf\!\left(\!\left(\pi\right)\!\right)$.
 Tout \'el\'ement $\alpha\in\Finf^*$ peut donc s'\'ecrire:
 \[\alpha=\sum_{i\geq j}c_i\,\pi^i\]
 o\`u $j:=\vinfprime(\alpha)$ et $c_i\in\FFinf$ pour tout entier $i\geq j$.
 En posant ${\sgn(\alpha):=c_j}$\label{notasgn} on d\'efinit alors, suivant la terminologie de Hayes,
 une \emph{fonction signe} sur~$\Finf$
 (\ie un morphisme de groupe ${\Finf^*\longrightarrow\FFinf^*}$ dont la restriction \`a~$\FFinf^*$ est l'identit\'e).

 Soit maintenant $\psi$ un $\OF$-module de Drinfeld de rang $1$. Pour tout $a\in\OF$,
 on note $\mu_\psi(a)\in\kbar$\label{notamupsia} le coefficient dominant de $\psi_a$,
 c'est-\`a-dire le coefficient du mon\^ome de plus haut degr\'e de $\psi_a\in\kbar\{\tau\}$.
 Pour tous $a\in\OF$ et $b\in\OF$ on a:
 \begin{equation}\label{eqrelmupsiab}
  \mu_\psi(a b)=\mu_\psi(a)\,\mu_\psi(b)^{q^{\degOF a}}\text{.}
 \end{equation}
 
 
 D'apr\`es le \S$6$ de \cite{HAYES}, l'application $x\longmapsto\mu_\psi(x)$ s'\'etend en une application
 \[
  \mu_\psi:
  \begin{array}[t]{ccl}
   F_{\infty'} & \longrightarrow & \C_\infty \\
   x & \longmapsto & \mu_\psi(x)
  \end{array}\text{.}
 \]

 \begin{defi}
  Soit $\psi$\label{notapsi} un $\OF$-module de Drinfeld de rang $1$.
  On dit que $\psi$ est de \emph{signe normalis\'e}
  s'il existe $\sigma\in\Gal(\FFinf/\F_q)$ tel que:
  \[\mu_\psi=\sigma\circ\sgn\text{.}\]
 \end{defi}

 En particulier, si $\psi$ est de signe normalis\'e, alors pour tout $a\in\OF\setminus\{0\}$, on~a:
 \begin{equation}\label{eqmupsidansFqdstar}
  \mu_\psi(a)\in\FFinf^*\simeq\F_{q^d}^*\text{,}
 \end{equation}
 car $F$ est de type CM.

 D'apr\`es le th\'eor\`eme $12.3$ de \cite{HAYESG},
 tout $\OF$-module de Drinfeld $\varphi$ de rang $1$ est isomorphe \`a un module de Drinfeld de signe normalis\'e.
 Plus pr\'ecis\'ement:
 
\begin{prop}\label{propracinepolysgnnormalise}
 Soient $\varphi$ un $\OF$-module de Drinfeld de rang $1$ et $z\in\kbar$ une racine du polyn\^ome:
 \[X^{q^d-1}-\mu_\varphi\!\left(\pi^{-1}\right)\text{.}\]
 Alors le $\OF$-module de Drinfeld $\psi:=z\,\varphi\,z^{-1}$ est de signe normalis\'e.
\end{prop}

\begin{proof}
 Voir le th\'eor\`eme $12.3$ de \cite{HAYESG}.
\end{proof}

 Soient $\psi$ un $\OF$-module de Drinfeld de signe normalis\'e
 et $H_\psi$ son corps des coefficients (\cf \S\ref{subsectionModDrinfeld}).
 Comme $\mu_\psi(a)\in\FFinf$ pour tout $a\in\OF$,
 d'apr\`es la proposition~$11.3$ de \cite{HAYESG}, $\psi$ est d\'efini sur $\Ocal_{H_\psi}$.
 
 Par ailleurs, d'apr\`es le \S$14$ de \cite{HAYESG}, le corps $H_\psi$ est ind\'ependant du $\OF$-module de Drinfeld de signe normalis\'e,
 on le note $\HFplus$\label{notaHFplus}.
 De plus, l'extension $\HFplus/F$ est ab\'elienne, finie (\cf prop. $14.1$ de \cite{HAYESG})
 et non ramifi\'ee au-dessus de toutes les places finies de $F$ (\cf Th. $14.4$ de \cite{HAYESG}).
 
 On peut donc associer \`a tout id\'eal premier non nul $\mcar$\label{notamcar}
 de $\OF$ son \emph{symbole d'Artin},
 c'est-\`a-dire l'unique automorphisme $\sigma_\mcar\in\Gal(\HFplus/F)$\label{notasigmamcar} v\'erifiant:
 \begin{equation*}
  \forall\gamma\in\OcalHFplus,\quad\sigma_\mcar(\gamma)\equiv\gamma^{q^{\degOF\mcar}}\mod\ncar
 \end{equation*}
 o\`u $\ncar$ est un id\'eal premier de $\OcalHFplus$ au-dessus de $\mcar$.
 Soient $\acar$ un id\'eal non nul de $\OF$
 et $\acar=\mcar_1^{e_1}\cdots\mcar_j^{e_j}$ sa d\'ecomposition en produit d'id\'eaux premiers de $\OF$.
 On \'etend la d\'efinition du symbole d'Artin \`a tout id\'eal non nul $\acar$ de $\OF$ en posant:
 \begin{equation*}
  \sigma_\acar=\sigma_{\mcar_1}^{e_1}\cdots\sigma_{\mcar_j}^{e_j}\text{.}
 \end{equation*}
 
 Soit $\HFplus(\psi[\acar])$\label{notaHFpluspsiacar}
 le corps engendr\'e sur $\HFplus$ par les points de $\acar$-torsion de~$\psi$.
 Alors, d'apr\`es la proposition $7.5.4$ de \cite{GOSS}, l'extension $\HFplus(\psi[\acar])/F$ est ab\'elienne
 et non ramifi\'ee au-dessus de tout id\'eal premier de~$\OF$ qui ne divise pas $\acar$.
 Dans le cas des id\'eaux premier de $\OF$ nous avons le r\'esultat compl\'ementaire suivant (\cf prop. $7.5.18$ de \cite{GOSS}):
 
 \begin{prop}\label{propextpttorsiontotalrami}
  Soient $\mcar$ un id\'eal premier non nul de $\OF$ et ${i>0}$ un entier.
  Alors, l'extension $\HFplus(\psi[\mcar^i])/\HFplus$ est totalement ramifi\'ee en tout premier de $\OcalHFplus$ au-dessus de $\mcar$.
 \end{prop}

 Soient $\lambda\in\kbar$ un point de $\acar$-torsion, $\bcar$ un id\'eal non nul de $\OF$ premier \`a $\acar$
 et $\sigma_\bcar\in\Gal\!\left(\HFplus(\psi[\acar])/F\right)$ le symbole d'Artin de $\bcar$.
 Alors, toujours d'apr\`es la proposition $7.5.4$ de \cite{GOSS}, on a:
 \begin{equation}\label{eqsymboleArtinidealpttors}
  \sigma_\bcar(\lambda)=\psi_\bcar(\lambda)\text{.}
 \end{equation}
 De plus, d'apr\`es le corollaire $7.5.6$ de \cite{GOSS}, on a:
 \begin{equation}\label{eqisogrpGaloisquotientcroix}
  \Gal\!\left(\HFplus(\psi[\acar])/\HFplus\right)\simeq\left(\OF/\acar\right)^\times\text{.}
 \end{equation} 

 Soient $\mcar$ un id\'eal premier non nul de $\OF$ et $\ncar$
 un id\'eal premier non nul de $\OcalHFplus$ au-dessus de $\mcar$ (\ie $\mcar=\ncar\cap\OF$). 
 D'apr\`es le corollaire $5.9$ de \cite{HAYESG}, pour tout entier $i\geq1$ on a:
 \begin{equation}\label{eqcongruencepsiidealpremier}
  \psi_{\mcar^i}\equiv\tau^{i\,\degOF\mcar}\mod\ncar
 \end{equation}
 (rappelons que $\psi_{\mcar^i}$ d\'esigne le g\'en\'erateur unitaire de l'id\'eal \`a gauche de~$\kbar\{\tau\}$
 engendr\'e par la famille $\left\{\psi_a\right\}_{a\in\mcar^i}$).

 Nous voudrions maintenant pouvoir nous servir de cette congruence
 pour obtenir des minorations utiles par la suite.
 Le probl\`eme avec cette congruence est que $\psi_{\mcar^i}$ n'est pas, en g\'en\'eral, un endomorphisme de $\psi$.
 Pour que $\psi_{\mcar^i}$ soit un endomorphisme de $\psi$ il faut et il suffit qu'il existe $m\in\OF$ tel que $\psi_m=\psi_{\mcar^i}$.
 Il suffit donc que l'id\'eal $\mcar^i$ soit principal
 et que l'un de ses g\'en\'erateurs $m\in\OF$ v\'erifie $\mu_\psi(m)=1$
 (puisque si $\mcar^i=(m)$, on a $\psi_m=\mu_\psi(m)\,\psi_{\mcar^i}$).

 L'id\'ee est donc de consid\'erer une puissance suffisamment grande de l'id\'eal $\mcar$.
 Pour cela,
 On note~$\nbclass$\label{notahnbclasses}
 le \emph{nombre de classes}
 de $\OF$ et $d:=\deg\infty'$ le degr\'e de la place $\infty'$.
 Alors $\mcar^\nbclass$ est un id\'eal principal.
 On fixe $m\in\OF$\label{notam} un g\'en\'erateur de $\mcar^\nbclass$. Pour tout entier $i>0$,
 on a d'apr\`es (\ref{eqrelmupsiab}) p. \pageref{eqrelmupsiab}:
 \[\mu_\psi(m^i)=\mu_\psi(m)^{1+q^{\degOF m}+\cdots+q^{(i-1)\degOF m}}\text{.}\]
 Or $\mu_\psi(m)\in\F_{q^d}^{\times}$ car $\psi$ est de signe normalis\'e (\cf (\ref{eqmupsidansFqdstar}) p. \pageref{eqmupsidansFqdstar}).
 Donc si $j$ et $j'$ repr\'esentent respectivement le quotient et le reste de la division euclidienne de $i$ par~$d$, alors:
 \[\mu_\psi(m)^{q^{i\degOF m}}
  =\mu_\psi(m)^{q^{(j d+j')\degOF m}}=\mu_\psi(m)^{q^{j'\degOF m}}\text{.}\]
 Ainsi,
 \[\mu_\psi(m^{i\,d})=\left(\mu_\psi(m)^{1+q^{\degOF m}+\cdots+q^{(d-1)\degOF m}}\right)^i\text{.}\]
 En prenant $i=q^d-1$, on obtient:
 \[\mu_\psi(m^{d(q^d-1)})=\left(\mu_\psi(m)^{1+q^{\degOF m}+\cdots+q^{(d-1)\degOF m}}\right)^{q^d-1}=1\text{.}\]

 Nous consid\'ererons plus loin la situation o\`u
 $K/\HFplus$ et $L/K$ sont deux extensions finies, $\pcar$ est un id\'eal premier de~$\OK$ au-dessus de~$\mcar$
 et $\qcar$ un id\'eal premier de $\OL$ au-dessus de~$\pcar$.
 Posons 
 \[\nu:=d(q^d-1)\degOL\qcar/\degOF\mcar\]
 (c'est un entier car $\OL/\qcar$ est une extension finie du corps fini $\OF/\mcar$).
 Alors, d'apr\`es la congruence {(\ref{eqcongruencepsiidealpremier}) p.~\pageref{eqcongruencepsiidealpremier}}
 et la discussion ci-dessus, on a:
 \[\psi_{m^\nu}=\psi_{\mcar^{\nbclass\nu}}\equiv\tau^{\nbclass\,d(q^d-1)\degOL\qcar}\mod\pcar\text{.}\]
 Cela montre:

 \begin{prop}\label{propCongModDNormalise}
  Soient $F/k$ une extension CM finie,
  $\infty'$ l'unique place de $F$ au-dessus $\infty$,
  $d$ le degr\'e sur $\F_q$ de la place $\infty'$,
  $\nbclass$ le nombre de classes de $\OF$
  et $\psi$ un $\OF$-module de Drinfeld de signe normalis\'e.
  Soient~$\mcar$ un id\'eal premier non nul de $\OF$,
  $m\in\OF$ un g\'en\'erateur de l'id\'eal principal~$\mcar^\nbclass$,
  $K/\HFplus$ et $L/K$ deux extensions finies,
  $\pcar$ un id\'eal premier de $\OK$ au-dessus de~$\mcar$
  et $\qcar$ un id\'eal premier de $\OL$ au-dessus de~$\pcar$.
  On pose:
  \[\nu:=d\,(q^d-1)\degOL\qcar/\degOF\mcar\text{.}\]
  Alors on a:
  \begin{equation}\label{eq:propCMDN}
   \psi_{m^\nu}=\psi_{\mcar^{\nbclass\,\nu}}\equiv\tau^{\nbclass\,d\,(q^d-1)\degOL\qcar}\mod\pcar\text{.}
  \end{equation}
 \end{prop}


\subsection{Hauteurs}
\leavevmode\par

 Soient $\alpha\in\kbar$ et $K$ une extension finie de $k$ contenant $\alpha$.
 La \emph{hauteur logarithmique et absolue de Weil}
 de $\alpha$, not\'ee $\hw(\alpha)$\label{notahw}, est d\'efinie par:
 \[\hw(\alpha):=\frac{1}{[K:k]}\sum_{v\in\Mcal_K}\deg v\,\log_q\max\{1,|\alpha|_v\}\text{.}\]
 Cette d\'efinition ne d\'epend pas de l'extension finie $K$ choisie.
 Rappelons que la hauteur de Weil est une fonction positive qui ne s'annule qu'en z\'ero et les racines de l'unit\'e
 et qu'elle est invariante par conjugaison
 (\ie pour tout $\alpha\in\kbar$ et tout conjugu\'e $\beta\in\kbar$ de $\alpha$ sur $k$, on a $\hw(\beta)=\hw(\alpha)$).
 Par ailleurs, elle v\'erifie
 la \emph{propri\'et\'e de Northcott}:
 pour tous r\'eels $D>0$ et $c>0$,
 il n'existe qu'un nombre fini de points $\alpha\in\kbar$ tels que $[k(\alpha):k]\leq D$ et $\hw(\alpha)\leq c$.

 De plus, pour tous $\alpha,\beta\in\kbar$ et pour tout $n\in\Z$, on~a:
 \begin{itemize}
  \item[$\minibullet$] $\hw(\alpha+\beta)\leq \hw(\alpha)+\hw(\beta)$;
  \item[$\minibullet$] $\hw(\alpha\,\beta)\leq \hw(\alpha)+\hw(\beta)$;
  \item[$\minibullet$] $\hw(\alpha^n)=|n|\,\hw(\alpha)$ o\`u $|\cdot|$\label{notavaleurabsolueusuelle}
   est la valeur absolue usuelle sur $\R$.
 \end{itemize}


 L. Denis (\cf \cite{DENIS1}) a associ\'e \`a tout $A$-module de Drinfeld une hauteur canonique.
 La construction de L. Denis 
 se g\'en\'eralise sans aucune difficult\'e aux \mbox{$R$-modules} de Drinfeld d\'efinis sur $\kbar$,
 o\`u $R$ est un ordre d'une extension CM finie de~$k$. On se contente donc d'en rappeler sa d\'efinition:
 
 \begin{defi}
  Soient $F/k$ une extension CM finie, $R$ un ordre de $F$
  et $\phi'$ un $R$-module de Drinfeld de rang $r'$ d\'efini sur $\kbar$.
  On appelle \emph{hauteur canonique} de $\phi'$
  la fonction $\hcan_{\phi'}:\kbar\longrightarrow\R_+$\label{notahcanphiprime}
  d\'efinie~par:
  \[\hcan_{\phi'}(\alpha):=\lim\limits_{i\to\infty}\frac{\hw\!\left(\phi'_{a^i}(\alpha)\right)}{q^{i\,r'\degR a}}\]
  o\`u $a\in R\setminus\Fbar_q$.
 \end{defi}
 
 Cette d\'efinition ne d\'epend pas de l'\'el\'ement $a$ choisi (th. $1$ de \cite{DENIS1}).

 Comme dans le cas classique, la hauteur canonique satisfait les bonnes propri\'et\'es suivantes:

 \begin{theo}\label{hautcan}
  Soient $F/k$ une extension CM finie, $R$ un ordre de~$F$
  et $\phi'$ un $R$-module de Drinfeld de rang $r'$ d\'efini sur $\kbar$.
  Alors la hauteur canonique $\hcan_{\phi'}$ de $\phi'$ v\'erifie les propri\'et\'es suivantes: 
  \begin{enumerate}
   \item Pour tout $a\in R$ et tout $\alpha\in\kbar$, on a:
    \[\hcan_{\phi'}\!\left(\phi'_a(\alpha)\right)=q^{r'\degR a}\ \hcan_{\phi'}(\alpha)\text{;}\]
   \item Il existe une constante r\'eelle $\gamma(\phi')\geq1$ telle que pour tout $\alpha\in\kbar$, on~a:
    \[\left|\hcan_{\phi'}(\alpha)-\hw(\alpha)\right|\leq\gamma(\phi')\text{;}\]
   \item La fonction $\hcan_{\phi'}$ est l'unique fonction d\'efinie sur $\kbar$ v\'erifiant \`a la fois la propri\'et\'e~$1$
    pour un certain $a\in R\setminus\Fbar_q$ et la propri\'et\'e $2$;
   \item La fonction $\hcan_{\phi'}$ v\'erifie la propri\'et\'e de Northcott. 
   \item Soit $\alpha\in\kbar$, alors $\hcan_{\phi'}(\alpha)=0$ si et seulement si $\alpha$ est un point de torsion de $\phi'$;
   \item Pour tous $\alpha\in\kbar$ et $\delta\in\phi'\!\left(\kbar\right)_\tors$, on a:
    \[\hcan_{\phi'}(\alpha+\delta)=\hcan_{\phi'}(\alpha)\text{;}\]
   \item Pour tous $\alpha,\beta\in\kbar$, on a:
    \[\hcan_{\phi'}(\alpha+\beta)\leq\hcan_{\phi'}(\alpha)+\hcan_{\phi'}(\beta)\text{.}\] 
  \end{enumerate}
 \end{theo}



 Soient $F/k$ une extension CM finie, $R$ un ordre de $F$,
 $\phi'$ un $R$-module de Drinfeld de rang $r'$ d\'efini sur $\kbar$
 et $\phi$ la restriction de $\phi'$ \`a $A$ (\cf \'egalit\'e (\ref{eqextscalairesModDrinfeld}) p. \pageref{eqextscalairesModDrinfeld}).
 Nous allons maintenant montrer que les hauteurs canoniques relatives \`a~$\phi$ et~$\phi'$ co\"incident.

 \begin{prop}\label{propEgalHautCanparExtScal}
  Pour tout $\alpha\in\bar{k}$, on a $\hcan_\phi(\alpha)=\hcan_{\phi'}(\alpha)$.
 \end{prop}

 \begin{proof}
  Soit $t$ le rang de $R$ en tant que $A$-module. 
  D'apr\`es le~\S\ref{subsectionEndoModDrinfeld},
  $\phi$ est de rang $r=r'\,t$ et pour tout $a\in A$, on a $\deg_R a=t\,\deg_A a$.
  Soit $a\in A$ tel que $a\in R\setminus\Fbar_q$.
  Alors pour tout $\alpha\in\bar{k}$ on a:
  \begin{multline*}
   \hcan_{\phi'}(\alpha)=\lim\limits_{n \to +\infty}\frac{\hw(\phi'_{a^n}(\alpha))}{q^{n r'\deg_R a}}
    =\lim\limits_{n \to +\infty}\frac{\hw(\phi_{a^n}(\alpha))}{q^{n r' t\deg_A a}} \\
    =\lim\limits_{n \to +\infty}\frac{\hw(\phi_{a^n}(\alpha))}{q^{n r\deg_A a}}
    =\hcan_{\phi}(\alpha)\text{.}
  \end{multline*}
 \end{proof}

 Par abus de notation on pourra encore noter $\hcan_\phi$ la hauteur normalis\'ee du $R$-module de Drinfeld~$\phi'$.

 On vient de voir que la hauteur canonique est invariante par extension des scalaires du module de Drinfeld.
 On va maintenant montrer par le biais de la proposition suivante,
 qui est une g\'en\'eralisation de la proposition $2$ de \cite{POONEN},
 que l'on peut relier les hauteurs canoniques associ\'ees \`a deux modules de Drinfeld isog\`enes.

 \begin{prop}\label{prophcanetisog}\!
  Soient\! $\rho$\! et\! $\varrho$ deux $R$-modules de Drinfeld de rang~$r$ d\'efinis sur $\bar{k}$ isog\`enes via $P\in\bar{k}\{\tau\}$
  (\emph{i.e.}  $P\,\rho=\varrho\,P$).
  Alors, pour tout $\alpha\in\bar{k}$, on a:
  \[q^{\deg_\tau P}\,\hcan_\rho(\alpha)=\hcan_\varrho\!\left(P(\alpha)\right)\text{.}\]
 \end{prop}

 \begin{proof}
  D'apr\`es la propri\'et\'e fonctorielle de la hauteur (\emph{cf.} th. $B.2.5$ de \cite{HINSIL}),
  il existe une constante $c(P)>0$ ne d\'ependant que du polyn\^ome~$P$ telle que pour tout $\alpha\in\bar{k}$, on a:
  \[\left|\hw\!\left(P(\alpha)\right)-q^{\deg_\tau P}\,\hw(\alpha)\right|\leq c(P)\text{.}\]
  Ainsi, pour tout $a\in R\setminus\overline{\F}_q$ et tout entier $i>0$, on a:
  \[\left|\frac{\hw\!\left(P\!\left(\rho_{a^i}(\alpha)\right)\right)}{q^{i\,r\deg_R a}}
   -\frac{q^{\deg_\tau P}\,\hw\!\left(\rho_{a^i}(\alpha)\right)}{q^{i\,r\deg_R a}}\right|\leq\frac{c(P)}{q^{i\,r\deg_R a}}\text{.}\]
  Or $P\!\left(\rho_{a^i}(\alpha)\right)=\varrho_{a^i}\!\left(P(\alpha)\right)$, d'o\`u:
  \[\left|\frac{\hw\!\left(\varrho_{a^i}\!\left(P(\alpha)\right)\right)}{q^{i\,r\deg_R a}}
   -\frac{q^{\deg_\tau P}\,\hw\!\left(\rho_{a^i}(\alpha)\right)}{q^{i\,r\deg_R a}}\right|\leq\frac{c(P)}{q^{i\,r\deg_R a}}\text{.}\]
  On obtient donc le r\'esultat souhait\'e en faisant tendre $i$ vers l'infini.
 \end{proof}

\section{D\'emonstration du th\'eor\`eme \ref{thlehmerBpsifini}}\label{chapB}\label{CHAPB}

Le but de cette partie est de d\'emontrer le th\'eor\`eme \ref{thlehmerBpsifini}.
Pour ce faire,
nous commencerons par montrer dans le~\S\ref{inegultrametriqB}
des in\'egalit\'es ultram\'etriques faisant intervenir un module de Drinfeld de signe normalis\'e,
ce qui nous permettra, gr\^ace \`a la formule du produit, 
d'en d\'eduire au \S\ref{minorhautcanBcondi} des minorations conditionnelles de la hauteur canonique.
\`A l'aide d'un argument reposant entre autres sur le principe des tiroirs (\cf \ref{principtiroirsBannquotient}),
nous nous affranchirons dans le~\S\ref{demothpsifinie} de ces contraintes.


\label{notademo}
\`A partir de maintenant et jusqu'au \S\ref{minorhautcanBcondi}, on fixe 
$F/k$ une extension CM finie, $\psi$ un $\OF$-module de Drinfeld de signe normalis\'e,
$K/k$ une extension finie sur laquelle $\psi$ est d\'efini
et $L/K$ une extension galoisienne finie de groupe de Galois $G$.
Soient $H$ un sous-groupe du centre de $G$ et $E\subseteq L$ le sous-corps fix\'e par $H$.
On se donne encore $d_1>0$ un entier et $v$ une place finie non triviale de $K$ telle que pour toute place $w|v$ de~$E$,
on ait $\left[E_w:K_v\right]\leq d_1$.
on note $\pcar$ l'id\'eal premier de $\OK$ associ\'e \`a la place $v$,
$\mcar:=\pcar\cap\OF$ l'id\'eal premier de $\OF$ en-dessous de $\pcar$
et on fixe $\qcar$ un id\'eal premier de $\Oe$ au-dessus de $\pcar$ (\ie $\pcar=\qcar\cap\OK$).
Soient $\nbclass$ le nombre de classes de $\OF$, 
$m\in\OF$ un g\'en\'erateur de l'id\'eal principal $\mcar^\nbclass$
et $d$ le degr\'e sur $\F_q$ de l'unique place de $F$ au-dessus de $\infty$.
On pose:
\begin{equation}\label{eqdefnu}
\nu:=d\,(q^d-1)\degOE\qcar/\degOF\mcar\text{;}
\end{equation}
\begin{equation}\label{eqdefn}
n:=\nbclass\,d\,(q^d-1)\degOE\qcar\text{;}
\end{equation}
\textit{et}
\begin{equation}\label{eqdefs}
s:=\nbclass\,d\,(q^d-1)\text{.}
\end{equation}

\subsection{In\'egalit\'es m\'etriques}\label{inegultrametriqB}
\leavevmode\par

On \'etablit ici quelques in\'egalit\'es ultram\'etriques utiles pour la suite.

\begin{lemm}\label{lemclef1}
 Soit $M$\! une extension finie de $K$.
 Alors, 
 pour tout ${\alpha\!\in\! M}$ et toute place finie $w$ de $M$, on~a:
 \[\max\!\left\{1,\left|\psi_{m^\nu}(\alpha)\right|_w\right\}
  =\max\!\left\{1,\left|\alpha\right|_w^{q^n}\right\}\text{.}\]
\end{lemm}

\begin{proof}
 Soit $\alpha\in M$. Comme $\psi$ est de signe normalis\'e, ses coefficients sont entiers (\emph{cf.}~\S\ref{subsectionThHayes}).
 De plus, $\psi_{m^\nu}$ est un polyn\^ome unitaire de degr\'e $n$ d'apr\`es la proposition \ref{propCongModDNormalise} (relation (\ref{eq:propCMDN})).
 Il existe donc des \'el\'ements ${a_0,\dots,a_{n-1}\in\OK}$ tels que:
 \[\psi_{m^\nu}(\alpha)=a_0\,\alpha+\cdots+a_{n-1}\,\alpha^{q^{n-1}}+\alpha^{q^n}\text{.}\]
 De plus, comme pour tout $i\in\{0,\dots,n-1\}$ le coefficient $a_i$ est entier, pour toute place finie $w$ de~$M$, on a $|a_i|_w\leq1$.
 Ainsi, il y a deux cas possibles:
 \begin{enumerate}
  \item Si $|\alpha|_w>1$, alors pour tout $i\in\{0,\dots,n-1\}$, on a:
   \[\left|a_i\,\alpha^{q^i}\right|_w\leq|\alpha|_w^{q^i}<|\alpha|_w^{q^n}\text{,}\]
   d'o\`u
   \[\left|\psi_{m^\nu}(\alpha)\right|_w=|\alpha|_w^{q^n}>1\]
   et donc
   \[\max\!\left\{1,\left|\psi_{m^\nu}(\alpha)\right|_w\right\}
    =|\alpha|_w^{q^n}=\max\!\left\{1,\left|\alpha\right|_w^{q^n}\right\}\text{.}\]
  \item Si $|\alpha|_w\leq1$, alors:
   \[\left|\psi_{m^\nu}(\alpha)\right|_w\leq1\] 
   et donc
   \[\max\!\left\{1,\left|\psi_{m^\nu}(\alpha)\right|_w\right\}
    =1=\max\!\left\{1,\left|\alpha\right|_w^{q^n}\right\}\text{.}\]
 \end{enumerate}
\end{proof}

\begin{lemm}\label{lemclef2}
 Soit $M$\! une extension finie de $K$.
 Alors, 
 pour tout ${\alpha\!\in\! M}$ et toute place $w|v$ de $M$, on~a:
 \[\left|\psi_{m^\nu}(\alpha)-\alpha^{q^n}\right|_w\leq q^{-e(w/v)}\,\max\!\left\{1,|\alpha|_w\right\}^{q^{n-1}}\text{.}\]
\end{lemm}

\begin{proof}
 Soient $a_0,\dots,a_{n-1}$ les coefficients de $\psi_{m^\nu}$, \ie:
 \[\psi_{m^\nu}=a_0\,\tau^0+\cdots+a_{n-1}\,\tau^{n-1}+\tau^n\text{.}\]
 Soient $\alpha\in M$ et $w$ une place de $M$ au-dessus de $v$, alors, d'apr\`es l'in\'egalit\'e ultram\'etrique, on a:
 \[\left|\psi_{m^\nu}(\alpha)-\alpha^{q^n}\right|_w\leq\max_{0\leq i\leq n-1}\!\left\{|a_i|_w\,|\alpha|_w^{q^i}\right\}\text{.}\]
 Or, d'apr\`es la congruence (\ref{eq:propCMDN}) de la proposition \ref{propCongModDNormalise}, on a:
 \[\psi_{m^\nu}\equiv\tau^n\mod\mathfrak{p}\text{.}\]
 Ainsi pour tout $i\in\{0,\dots,n-1\}$, on a $a_i\in\mathfrak{p}$, d'o\`u $|a_i|_w\leq q^{-e(w/v)}$ ce qui permet de conclure.
\end{proof}

Le lemme suivant, bas\'e sur le th\'eor\`eme d'approximation forte, nous permettra de nous ramener dans les preuves au cas des entiers.

\begin{lemm}\label{lemSAT}
 Soit $M$ une extension finie de $k$.
 Soit $w$ une place finie de $M$.
 Alors, pour tout $\alpha\in M$ non nul,
 il existe $\beta\in\OM$ tel que ${\alpha\,\beta\in\OM}$,
 et $|\beta|_w=\max\{1,|\alpha|_w\}^{-1}$.
\end{lemm}

\begin{proof}
 Soit $S$ l'ensemble des places finies de $M$ (\emph{i.e.} des places $\omega$ de $M$ telles que $\omega\nmid\infty$).
 Soit $\Sigma_0:=\left\{\omega\in S\,\big|\,|\alpha|_\omega>1\right\}$ (c'est un ensemble fini car $\alpha$ n'a qu'un nombre fini de p\^oles).
 Soit $\Sigma:=\Sigma_0\cup\{w\}$.
 Alors, d'apr\`es le th\'eor\`eme d'approximation forte (\emph{cf.} th. $1.6.5$ de \cite{STICHT}),
 il existe $\beta\in M$ tel que $|\beta|_\omega=\max\{1,|\alpha|_\omega\}^{-1}$ pour toute $\omega\in\Sigma$
 et $|\beta|_\omega\leq1$ pour toute $\omega\in S\setminus\Sigma$.
 Ainsi $\beta$ et~$\alpha\,\beta$ sont bien dans $\OM$.
\end{proof}



\begin{prop}\label{propclefextBnonrami}
 Supposons que l'extension $L/E$ soit non ramifi\'ee au-dessus de~$\qcar$.
 Soit $\sigma_\qcar$ le symbole d'Artin associ\'e \`a $\qcar$.
 Alors, pour toute place non triviale $w$ de $L$ et tout $\alpha\in L$, on a:
 \[\left|\psi_{m^\nu}(\alpha)-\sigma_\qcar^s(\alpha)\right|_w
  \leq C(w)\,\max\!\left\{1,\left|\psi_{m^\nu}(\alpha)\right|_w\right\}
  \,\max\!\left\{1,|\sigma_\qcar^s(\alpha)|_w\right\}\]
 o\`u $C(w)=q^{-1}$ si $w|v$ et $C(w)=1$ sinon. 
\end{prop}

\begin{proof}
 Soit $w$ une place non triviale de $L$.
 Si $w\nmid v$, l'in\'egalit\'e d\'ecoule directement de l'in\'egalit\'e ultram\'etrique.
 On suppose donc que~$w|v$. 
 Pour tout $\gamma\in\OL$, on a:
 \[\sigma_\qcar(\gamma)\equiv\gamma^{q^{\degOE\qcar}}\!\!\mod\qcar\OL\]
 et donc
 \begin{equation}\label{eq:congsymbArtin}
  \sigma_\qcar^s(\gamma)\equiv\gamma^{q^{s\,\degOE\qcar}}\equiv\gamma^{q^n}\!\!\mod\qcar\OL
 \end{equation}
 car $n=s\,\degOE\qcar$ (\cf notations (\ref{eqdefn}) et (\ref{eqdefs}) p. \pageref{eqdefs}).
 Si maintenant $\qcar'$ est un autre id\'eal de $\Oe$ au-dessus de $\pcar$,
 alors $\qcar'$ est non ramifi\'e dans $L$ et $\degOE\qcar'=\degOE\qcar$
 (car $L/K$ et $E/K$ sont galoisiennes).
 On a donc encore la congruence (\ref{eq:congsymbArtin}) avec $\qcar'$ au lieu de $\qcar$.
 Mais $\sigma_\qcar$ et $\sigma_{\qcar'}$,
 sont des \'el\'ements de $H$ conjugu\'es dans $G$ et $H$ est dans $\Zcal(G)$, donc $\sigma_\qcar=\sigma_{\qcar'}$.
 Ceci montre que pour tout $\gamma\in\OL$, on a:
 \[\left|\gamma^{q^n}-\sigma_\qcar^s(\gamma)\right|_w\leq q^{-1}\text{.}\]
 Soit $\alpha\in L$, alors d'apr\`es le lemme \ref{lemSAT},
 il existe $\beta\in\mathcal{O}_L$ tel que $\alpha\,\beta\in\OL$ et $|\beta|_w=\max\{1,|\alpha|_w\}^{-1}$.
 Ainsi:
 \begin{eqnarray*}
  \big|\alpha^{q^n}\!\!\!\!\!\!\!\!\!\! & - & \!\!\!\!\!\!\!\!\!\!\sigma_\qcar^s(\alpha)\big|_w \\
  & = & |\beta|_w^{-q^n}\,\left|(\alpha \beta)^{q^n}-\sigma_\qcar^s(\alpha \beta)
   +\left(\sigma_\qcar^s(\beta)-\beta^{q^n}\right)\sigma_\qcar^s(\alpha)\right|_w \\
  & \leq & |\beta|_w^{-q^n}
   \,\max\!\left\{\left|(\alpha \beta)^{q^n}-\sigma_\qcar^s(\alpha \beta)\right|_w,
   \left|\left(\sigma_\qcar^s(\beta)-\beta^{q^n}\right)\sigma_\qcar^s(\alpha)\right|_w \right\} \\
  & \leq & \max\{1,|\alpha|_w\}^{q^n}\,q^{-1}\,\max\!\left\{1,\left|\sigma_\mathfrak{p}^s(\alpha)\right|_w\right\} \\
  & \underset{\text{lem. \ref{lemclef1}}}{\leq} & q^{-1}\,\max\!\left\{1,\left|\psi_{m^{\nu}}(\alpha)\right|_w\right\}
   \,\max\!\left\{1,\left|\sigma_\qcar^s(\alpha)\right|_w\right\} \text{.}
 \end{eqnarray*}
 D'autre part,
 \begin{eqnarray*}
  \left|\psi_{m^\nu}(\alpha)-\alpha^{q^n}\right|_w
   & \underset{\text{lem. \ref{lemclef2}}}{\leq} & q^{-e(w/v)}\,\max\!\left\{1,|\alpha|_w\right\}^{q^{n-1}} \\
   & \underset{\text{lem. \ref{lemclef1}}}{\leq} & q^{-1}\,\max\!\left\{1,\left|\psi_{m^\nu}(\alpha)\right|_w\right\}\text{.}
 \end{eqnarray*}
 Le r\'esultat d\'ecoule maintenant directement de l'utilisation de l'in\'egalit\'e ultram\'etrique.
\end{proof}


La proposition suivante est l'analogue de la proposition $2.3$ de \cite{AMOZAN2},
la d\'emonstration \'etant quasi identique, on ne la fait pas.

\begin{prop}\label{grpHpnontrivialB}
 Supposons que l'extension $L/E$ soit ramifi\'ee au-dessus de~$\qcar$.
 Alors l'ensemble
 \[\label{notaHqcar}
  \Hqcar:=\left\{\sigma\in\Gal(L/E)\,\big|\,\forall\gamma\in\OL,
  \sigma(\gamma)^{q^{\degOE\qcar}}\equiv\gamma^{q^{\degOE\qcar}}\!\!\!\!\mod\qcar\OL\right\}
 \]
 est un sous-groupe non trivial du groupe d'inertie de $\qcar$ sur $L$.
\end{prop}


La proposition \ref{propclefextBnonrami} \'etablit une in\'egalit\'e ultram\'etrique pour les extensions non ramifi\'ees,
une in\'egalit\'e similaire dans le cas ramifi\'e fait l'objet de la proposition suivante.

\begin{prop}\label{propclefextBrami}
 On suppose que l'extension $L/E$ est ramifi\'ee au-des\-sus de~$\qcar$.
 Soient $\eta\in H_\qcar$ tel que $\eta\neq\id$ et $e_\qcar(L/E)$ l'indice de ramification de~$\qcar$ sur~$L$.
 Alors, pour toute place $w$ de $L$ et tout $\alpha\in L$, on~a:
 \begin{multline*}
  \big|\psi_{m^\nu}(\alpha)-\psi_{m^\nu}(\eta(\alpha))\big|_w \\
   \leq C(w)\,\max\!\left\{1,\left|\psi_{m^\nu}(\alpha)\right|_w\right\}
    \,\max\!\left\{1,\left|\psi_{m^\nu}(\eta(\alpha))\right|_w\right\}
 \end{multline*}
 o\`u $C(w)=q^{-e_\qcar(L/E)}$ si $w|v$ et $C(w)=1$ sinon. 
\end{prop}

\begin{proof}
 Soit $w$ une place de $L$. 
 Si $w\nmid v$, l'in\'egalit\'e d\'ecoule directement de l'in\'egalit\'e ultram\'etrique.
 On suppose donc que $w|v$.
 Soit~$\qcar'$ un id\'eal premier de $\Oe$ au-dessus de $\pcar$.
 Alors $H_\qcar$ et~$H_{\qcar'}$ sont dans $H$ et sont conjugu\'es dans $G$.
 Or $H$ est dans $\Zcal(G)$, donc $H_\qcar=H_{\qcar'}$.
 De plus, $\degOE\qcar=\degOE\qcar'$ et $e_\qcar(L/E)=e_{\qcar'}(L/E)$.
 En posant:
 \[n:=\nbclass\,d\,(q^d-1)\degOE\qcar\text{,}\]
 on a donc:
 \[\eta(\gamma)^{q^n}\equiv\gamma^{q^n}\!\!\mod\qcar'\OL\text{.}\]
 Ainsi, pour tout $\gamma\in\OL$, on a: 
 \[\left|\gamma^{q^n}-\eta(\gamma)^{q^n}\right|_w\leq q^{-e_\qcar(L/E)}\text{.}\]
 Soit $\alpha\in L$. Alors d'apr\`es le lemme \ref{lemSAT},
 il existe $\beta\in\OL$ tel que $\alpha\,\beta\in\OL$ et $|\beta|_w=\max\{1,|\alpha|_w\}^{-1}$, ainsi:
 \begin{eqnarray*}
  \big|\alpha^{q^n}\!\!\!\!\!\!\!\!\!\! & - & \!\!\!\!\!\!\!\!\!\!\eta(\alpha)^{q^n}\big|_w \\
   & = & |\beta|_w^{-q^n}\!\!\,\left|(\alpha \beta)^{q^n}-\eta(\alpha \beta)^{q^n}
     +\left(\eta(\beta)^{q^n}-\beta^{q^n}\right)\eta(\alpha)^{q^n}\right|_w \\
   & \leq & |\beta|_w^{-q^n}\,\max\!\left\{\left|(\alpha \beta)^{q^n}-\eta(\alpha \beta)^{q^n}\right|,
     \left|\left(\eta(\beta)^{q^n}-\beta^{q^n}\right)\eta(\alpha)^{q^n}\right|_w \right\} \\
   & \leq & \max\{1,|\alpha|_v\}^{q^n}\,q^{-e_\qcar(L/K)}\,\max\!\left\{1,\left|\eta(\alpha)\right|_w\right\}^{q^n} \\
   & \underset{\text{lem. \ref{lemclef1}}}{\leq} & q^{-e_\qcar(L/K)}\,\max\!\left\{1,\left|\psi_{m^{\nu}}(\alpha)\right|_w\right\}
    \,\max\!\left\{1,\left|\psi_{m^{\nu}}(\eta(\alpha))\right|_w\right\} \text{.}
 \end{eqnarray*}
 D'autre part,
 \begin{eqnarray*}
  \left|\psi_{m^\nu}(\alpha)-\alpha^{q^n}\right|_w
   & \underset{\text{lem. \ref{lemclef2}}}{\leq} & q^{-e(w/v)}\,\max\!\left\{1,|\alpha|_w\right\}^{q^{n-1}} \\
   & \underset{\text{lem. \ref{lemclef1}}}{\leq} & q^{-e_\qcar(L/E)}\,\max\!\left\{1,\left|\psi_{m^\nu}(\alpha)\right|_w\right\}\text{.}
 \end{eqnarray*}
 Le r\'esultat d\'ecoule maintenant directement de l'utilisation de l'in\'egalit\'e ultram\'etrique.
\end{proof}


Le lemme suivant, que l'on couplera aux deux lemmes pr\'ec\'edents par la suite,
permet d'obtenir des majorations suffisamment fines en composant par $\psi_{m^\nu}$,
ce qui sera tr\`es utile pour obtenir les minorations de hauteur souhait\'ees (\cf prop. \ref{propextBnonrami} et \ref{propextBrami}).

\begin{lemm}[d'acc\'el\'eration]\label{accconv}
 Soient $w$ une place de $L$ au-dessus de $v$
 et $\alpha,\beta\in L$. On suppose qu'il existe un entier $c>0$ tel que:
 \[|\alpha-\beta|_w\leq q^{-c\,e_\qcar(L/E)}\,\max\{1,|\alpha|_w\}\,\max\{1,|\beta|_w\}\text{.}\]
 Alors, pour tout entier naturel $l$, on a:
 \[\left|\psi_{m^{\nu l}}(\alpha-\beta)\right|_w\leq q^{-(c+l)\,e_\qcar(L/E)}\,\max\!\left\{1,\left|\psi_{m^{\nu l}}(\alpha)\right|_w\right\}
  \,\max\!\left\{1,\left|\psi_{m^{\nu l}}(\beta)\right|_w\right\}\text{.}\]
\end{lemm}

\begin{proof}
 On raisonne par r\'ecurrence sur $l$.
 Pour $l=0$ le r\'esultat est imm\'ediat.
 Montrons le cas $l=1$.
 Soient $a_0,\dots,a_n\in\kbar$ les coefficients de $\psi_{m^\nu}$, \emph{i.e.}:
 \[\psi_{m^\nu}=a_0\,\tau^0+\cdots+a_n\,\tau^n\text{,}\]
 avec $a_0=m^\nu$ et $a_n=1$.
 D'apr\`es la congruence (\ref{eq:propCMDN}) p. \pageref{eq:propCMDN}, on a:
 \[\psi_{m^\nu}=\psi_{\mcar^{\nbclass\nu}}\equiv\tau^n\mod\pcar\text{.}\]
 Ainsi, pour tout $i\in\{0,\dots,n-1\}$, on a $|a_i|_w\leq q^{-e(w/v)}$.
 Or,
 \[\left|\psi_{m^\nu}(\alpha-\beta)\right|_w\leq\max_{0\leq i\leq n}\left\{|a_i|_w\,|\alpha-\beta|_w^{q^i}\right\}\text{.}\]
 Il y a donc deux cas possibles:
 \begin{enumerate}
  \item Si $|\alpha-\beta|_w\geq1$, alors:
   \begin{eqnarray*}
    \big|\psi_{m^\nu}(\alpha\!\!\!\!\!\!\!\!\!\! & - & \!\!\!\!\!\!\!\!\!\!\beta)\big|_w \\
     & \leq & |\alpha-\beta|_w^{q^n} \\
     & \leq & q^{-c\,e_\qcar(L/E)\,q^n}\,\max\{1,|\alpha|_w\}^{q^n}\,\max\{1,|\beta|_w\}^{q^n} \\
     & \underset{\text{lem.\ref{lemclef1}}}{\leq} & q^{-(c+1)\,e_\qcar(L/E)}\,\max\!\left\{1,\left|\psi_{m^\nu}(\alpha)\right|_w\right\}
      \,\max\!\left\{1,\left|\psi_{m^\nu}(\beta)\right|_w\right\} \text{;}
   \end{eqnarray*}
  \item Si $|\alpha-\beta|_w<1$, alors:
   \begin{eqnarray*}
    \big|\psi_{m^\nu}(\alpha\!\!\!\!\!\!\!\!\!\! & - & \!\!\!\!\!\!\!\!\!\!\beta)\big|_w \\
     & \leq & q^{-e(w/v)}\,|\alpha-\beta|_w \\
     & \leq & q^{-e(w/v)}\,q^{-c\,e_\qcar(L/E)}\,\max\{1,|\alpha|_w\}\,\max\{1,|\beta|_w\} \\
     & \leq & q^{-(c+1)\,e_\qcar(L/E)}\,\max\{1,|\alpha|_w\}^{q^n}\,\max\{1,|\beta|_w\}^{q^n} \\
     & \underset{\text{lem.\ref{lemclef1}}}{=} & q^{-(c+1)\,e_\qcar(L/E)}\,\max\!\left\{1,\left|\psi_{m^\nu}(\alpha)\right|_w\right\}
      \,\max\!\left\{1,\left|\psi_{m^\nu}(\beta)\right|_w\right\} \text{.}
   \end{eqnarray*}
 \end{enumerate}
 Supposons maintenant $l>1$. Alors
 \[\psi_{m^{\nu l}}(\alpha-\beta)=\psi_{m^\nu}\,\psi_{m^{\nu (l-1)}}(\alpha-\beta)\text{.}\]
 Or, par l'hypoth\`ese de r\'ecurrence, on a:
 \begin{multline*}
  \left|\psi_{m^{\nu (l-1)}}(\alpha-\beta)\right|_w \\
   \leq q^{-(c+l-1)\,e_\qcar(L/E)}
    \max\!\left\{1,\left|\psi_{m^{\nu (l-1)}}(\alpha)\right|_w\right\}
     \max\!\left\{1,\left|\psi_{m^{\nu (l-1)}}(\beta)\right|_w\right\}\text{.}
 \end{multline*}
 Il ne reste donc plus qu'\`a substituer $\psi_{m^{\nu (l-1)}}(\alpha)$ \`a $\alpha$ et $\psi_{m^{\nu (l-1)}}(\beta)$ \`a $\beta$,
 ainsi qu'\`a faire agir $\psi_{m^\nu}$ sur ${\psi_{m^{\nu (l-1)}}(\alpha)-\psi_{m^{\nu (l-1)}}(\beta)}$
 et \`a utiliser ce que l'on a montr\'e dans le cas $l=1$.  
\end{proof}

\subsection{Minorations de la hauteur canonique}\label{minorhautcanBcondi}
\leavevmode\par

Dans tout ce paragraphe on utilise les notations fix\'ees au d\'ebut de cette partie (\cf p. \pageref{notademo}). 
Elles feront implicitement partie
des conditions de tous les \'enonc\'es de ce paragraphe.
Notons aussi:
\[\label{notac3}
 c_3:=3\,\nbclass\,\dinfprime\,(q^\dinfprime-1)\,\gamma(\psi)\text{,}
\]
o\`u $\gamma(\psi)$ est le r\'eel d\'efini dans le th\'eor\`eme \ref{hautcan}, point $2$.

La proposition suivante fournit une minoration de la hauteur normalis\'ee dans le cas o\`u
l'extension~$L/E$ n'est pas ramifi\'ee au-dessus de $\qcar$.

\begin{prop}\label{propextBnonrami}
 On suppose que l'extension $L/E$ est non ramifi\'ee au-dessus de $\qcar$,
 alors pour tout $\alpha\in L$ qui n'est pas de torsion pour~$\psi$, on~a:
 \[\hcan_\psi(\alpha)\geq\frac{q^{-c_3\,\deg v\,d_1^2\,[K:k]}}{2\,[K:k]}\text{.}\]
\end{prop}

\begin{proof}
 Soit $\alpha\in L$ non de torsion pour $\psi$.
 Commen\c{c}ons par montrer que pour tout $\sigma\in\Gal(L/E)$ et pour tout $l\in\N$,
 on a:
 \[\psi_{m^{\nu (l+1)}}(\alpha)\neq\psi_{m^{\nu l}}\!\left(\sigma(\alpha)\right)\text{.}\]
 En effet, si $\psi_{m^\nu}\!\left(\psi_{m^{\nu l}}(\alpha)\right)=\sigma\!\left(\psi_{m^{\nu l}}(\alpha)\right)$,
 alors pour tout entier $j\geq1$, on~a:
 \[\psi_{m^\nu}^j\!\left(\psi_{m^{\nu l}}(\alpha)\right)=\sigma^j\!\left(\psi_{m^{\nu l}}(\alpha)\right)\]
 car $\psi_{m^\nu}$ et $\sigma$ commutent.
 Soit $o(\sigma)$ l'ordre de $\sigma$ dans $\Gal(L/K)$, alors:
 \[0=\psi_{m^\nu}^{o(\sigma)}\!\left(\psi_{m^{\nu l}}(\alpha)\right)-\psi_{m^{\nu l}}(\alpha)
  =\psi_{m^{\nu (l+1) o(\sigma)}-m^{\nu l}}(\alpha)\text{.}\]
 Donc $\alpha$ est un point de torsion pour~$\psi$, ce qui contredit l'hypoth\`ese faite sur $\alpha$.
 En particulier, en notant $\sigma_\mathfrak{p}$ le symbole d'Artin associ\'e \`a $\mathfrak{p}$,
 on obtient:
 \[\psi_{m^{\nu l}}\!\left(\psi_{m^\nu}(\alpha)-\sigma_\mathfrak{p}^s(\alpha)\right)\neq0\text{.}\]
 On peut donc appliquer la formule du produit (\cf th. $3$ du \S$3$ du chap. \uppercase\expandafter{\romannumeral 12} de \cite{ARTIN}):
 \[0=\sum_{w\in\mathcal{M}_L}\deg w\,\log_q\left|\psi_{m^{\nu l}}\!\left(\psi_{m^\nu}(\alpha)-\sigma_\mathfrak{p}^s(\alpha)\right)\right|_w\text{.}\]
 Ainsi, en appliquant la proposition \ref{propclefextBnonrami} et le lemme \ref{accconv} avec
 \[l:=\left(2\,\gamma(\psi)\,[K:k]+1\right)[E_\qcar:K_\pcar]-1\text{,}\]
 on obtient:
 \begin{multline}\label{eqFdPnonramiB}
  0\leq\sum_{w|v}\deg w\,(-(l+1))
   +\sum_{w\in\mathcal{M}_L} \deg w\,\log_q\max\!\left\{1,\left|\psi_{m^{\nu (l+1)}}(\alpha)\right|_w\right\} \\ 
    +\sum_{w\in\mathcal{M}_L} \deg w\,\log_q\max\!\left\{1,\left|\psi_{m^{\nu l}}\!\left(\sigma_\mathfrak{p}^s(\alpha)\right)\right|_w\right\}\text{.}
 \end{multline}
 Or, pour toute place $u|v$ de $E$, on a:
 \[\sum_{w|u} \deg w=[L:E]\,\deg u=[L:E]\,\degOE\qcar\]
 car l'extension $L/E$ est non ramifi\'ee et l'extension $L/K$ est galoisienne.
 De plus,
 \begin{multline*}
  \sum_{u|v}(l+1)\,[L:E]\,\deg u \\
   =\left(2\,\gamma(\psi)\,[K:k]+1\right)\,\sum_{u|v}[E_\qcar:K_\pcar]\,[L:E]\,\degOE\qcar \\
  \geq\left(2\,\gamma(\psi)\,[K:k]+1\right)\,[L:K]\,\degOK\pcar \text{.}
 \end{multline*}
 En divisant les membres de l'in\'egalit\'e (\ref{eqFdPnonramiB}) par $[L:k]$, on obtient:
 \[0\leq-\frac{\left(2\,\gamma(\psi)\,[K:k]+1\right)\,\degOK\pcar}{[K:k]}
   +\hw\!\left(\psi_{m^{\nu (l+1)}}(\alpha)\right)
   +\hw\!\left(\psi_{m^{\nu l}}\!\left(\sigma_\mathfrak{p}^s(\alpha)\right)\right)\text{,}\]
 d'o\`u
 \[\frac{\left(2\,\gamma(\psi)\,[K:k]+1\right)\,\degOK\pcar}{[K:k]}\leq2\,\gamma(\psi)
  +\hcan_\psi\!\left(\psi_{m^{\nu (l+1)}}(\alpha)\right)+\hcan_\psi\!\left(\psi_{m^{\nu l}}(\alpha)\right)\]
 car pour tout $\alpha\in\bar{k}$, on a $\left|\hw(\alpha)-\hcan_\psi(\alpha)\right|\leq\gamma(\psi)$,
 et la hauteur de Weil est invariante sous l'action du groupe de Galois.
 Ainsi:
 \[\frac{\left(2\,\gamma(\psi)\,[K:k]+1\right)\,\degOK\pcar}{[K:k]}\leq2\,\gamma(\psi)+2\,q^{n\,(l+1)}\,\hcan_\psi(\alpha)\text{,}\]
 d'o\`u
 \[\frac{1}{[K:k]}\leq2\,q^{n\,(l+1)}\,\hcan_\psi(\alpha)\text{.}\]
 D'autre part,
 \begin{eqnarray*}
  n\,(l+1) & = & \nbclass\,\dinfprime\,(q^\dinfprime-1)\,\deg_{\Oe}\qcar\,\left(2\,\gamma(\psi)\,[K:k]+1\right)[E_\qcar:K_\pcar] \\
  & \leq & 3\,\nbclass\,\dinfprime\,(q^\dinfprime-1)\,\gamma(\psi)\,[E_\qcar:K_\pcar]\,\deg_{\Oe}\qcar\,[K:k] \\
  & \leq & 3\,\nbclass\,\dinfprime\,(q^\dinfprime-1)\,\gamma(\psi)\,\deg v\,d_1^2\,[K:k] \text{.}
 \end{eqnarray*}
 Ce qui termine la preuve.
\end{proof}

Le r\'esultat qui suit donne une minoration de la hauteur normalis\'ee dans le cas ramifi\'e
et utilise la notation $\Hqcar$ d\'efinie dans la proposition~\ref{grpHpnontrivialB}.

\begin{prop}\label{propextBrami}
 On\ \,suppose\ \,que\ \,l'extension\ \,$L/E$\ \,est\ \,ramifi\'ee au-dessus de~$\qcar$.
 Soit~${\eta\in\Hqcar}$ tel que $\eta\neq\id$.
 Alors, pour tout $\alpha\in L$ tel que
 ${\alpha-\eta(\alpha)\notin\psi\left[\mcar^\infty\right]}$, on a:
 \[\hcan_\psi(\alpha)\geq\frac{q^{-c_3\,\deg v\,d_1^2\,[K:k]}}{2\,[K:k]}\text{.}\]
\end{prop}

\begin{proof}
 Cette preuve est similaire \`a celle de la proposition \ref{propextBnonrami} \`a la diff\'erence qu'on utilise
 la proposition \ref{propclefextBrami} \`a la place de la proposition \ref{propclefextBnonrami}.
 Soit $\alpha\in L$ tel que $\alpha-\eta(\alpha)\notin\psi[\mcar^\infty]$.
 Alors, en posant:
 \[l:=\left(2\,\gamma(\psi)\,[K:k]+1\right)[E_\qcar:K_\pcar]-1\text{,}\]
 on a:
 \[\psi_{m^{\nu (l+1)}}(\alpha)-\psi_{m^{\nu (l+1)}}(\eta(\alpha))\neq0\text{.}\]
 On peut donc appliquer la formule du produit (\cf th. $3$ du \S$3$ du chap. \uppercase\expandafter{\romannumeral 12} de \cite{ARTIN}), \ie:
 \[0=\sum_{w\in\mathcal{M}_L}\deg w\,\log_q\left|\psi_{m^{\nu (l+1)}}(\alpha)-\psi_{m^{\nu (l+1)}}(\eta(\alpha))\right|_w\text{.}\]
 Ainsi, en appliquant la proposition \ref{propclefextBrami} et le lemme \ref{accconv},
 on obtient:
\begin{multline}\label{eqFdPramiB}
  0\leq\sum_{w|v}\deg w\,\left(-(l+1)\,e_\qcar(L/E)\right) \\
   +\sum_{w\in\mathcal{M}_L} \deg w\,\log_q\max\!\left\{1,\left|\psi_{m^{\nu (l+1)}}(\alpha)\right|_w\right\} \\
   +\sum_{w\in\mathcal{M}_L} \deg w\,
    \log_q\max\!\left\{1,\left|\psi_{m^{\nu(l+1)}}(\eta(\alpha))\right|_w\right\}\text{.}
\end{multline}
 Comme dans la preuve de la proposition \ref{propextBnonrami}, on a:
 \begin{multline*}
  \sum_{w|v}(l+1)\deg w\,e_\qcar(L/E)=\sum_{u|v}(l+1) [L:E]\,\deg u \\
   \geq\left(2\,\gamma(\psi)\,[K:k]+1\right)\,[L:K]\,\degOK\pcar \text{.}
 \end{multline*}
 Donc, en divisant les membres de l'in\'egalit\'e (\ref{eqFdPramiB}) par $[L:k]$, on obtient:
 \[0\!\leq\!-\frac{\left(2\,\gamma(\psi)\,[K:k]+1\right)\,\degOK\pcar}{[K:k]}+\hw\!\left(\psi_{m^{\nu (l+1)}}(\alpha)\right)
  +\hw\!\left(\psi_{m^{\nu (l+1)}}\!\left(\eta(\alpha)\right)\right)\text{,}\]
 d'o\`u
\begin{multline*}
 \frac{\left(2\,\gamma(\psi)\,[K:k]+1\right)\,\degOK\pcar}{[K:k]}
  \leq2\,\gamma(\psi)+2\,\hcan_\psi\!\left(\psi_{m^{\nu (l+1)}}(\alpha)\right) \\
  =2\,\gamma(\psi)+2\,q^{n\,(l+1)}\,\hcan_\psi(\alpha)\text{.}
\end{multline*}
 Ainsi, on obtient:
 \[\frac{1}{[K:k]}\leq2\,q^{n\,(l+1)}\,\hcan_\psi(\alpha)\text{,}\]
 ce qui ach\`eve la preuve car
 \[n\,(l+1)\leq3\,\nbclass\,\dinfprime\,(q^\dinfprime-1)\,\gamma(\psi)\,\deg v\,d_1^2\,[K:k]\text{.}\]
\end{proof}

Nous pouvons regrouper les deux propositions pr\'ec\'edentes en un seul th\'eor\`eme qui fournit
une minoration conditionnelle de la hauteur normalis\'ee.

\begin{theo}[de minoration conditionnelle]\label{thBminorcondi}
 Soient $F/k$ une extension CM finie
 et $\psi$ un $\OF$-module de Drinfeld de signe normalis\'e.
 Soient $K/k$ une extension finie telle que~$\psi$ soit d\'efini sur $K$,
 $L/K$ une extension galoisienne finie de groupe de Galois $G$,
 $H$ un sous-groupe du centre de~$G$ et $E\subseteq L$ le sous-corps fix\'e par $H$.
 Soient $d_1>0$ un entier et $v$ une place finie de $K$.
 On suppose que pour toute place $w|v$ de $E$, on ait $\left[E_w:K_v\right]\leq d_1$.
 Soient $\pcar$ l'id\'eal premier de $\OK$ associ\'e \`a la place $v$
 et $\mcar:=\pcar\cap\OF$ l'id\'eal premier de~$\OF$ en-dessous de~$\pcar$.
 Soit $\alpha\in L$ non de torsion pour $\psi$.
 On suppose que:
 \begin{equation}\label{cond1B}
  \forall\eta\in\Gal(L/E),\quad\eta(\alpha)-\alpha\notin\psi[\mcar^\infty]\setminus\{0\}\text{.}
 \end{equation}
 Alors:
 \[\hcan_\psi(\alpha)\geq\frac{q^{-c_3\,\deg v\,d_1^2\,[K:k]}}{2\,[K:k]}\text{.}\]
\end{theo}

\begin{proof}
 Soit $\alpha\notin\psi(L)_\tors$ un \'el\'ement de $L$ qui v\'erifie la condition (\ref{cond1B}).
 Soit $I$ le groupe de Galois de $L/E(\alpha)$. 
 Comme $H:=\Gal(L/E)$ est un sous-groupe du centre de $G$, $H$ est ab\'elien et est un sous-groupe normal de~$G$.
 De m\^eme, comme~$I$ est un sous-groupe de $H$, $I$ est ab\'elien et est un sous-groupe normal de $G$.
 Ainsi, $H/I$ est un sous-groupe normal de~$G/I$ et:
 \[(G/I)\big/(H/I)\simeq G/H\text{.}\]
 De plus, comme $H$ est inclus dans le centre de $G$, par passage au quotient $H/I$ est un sous-groupe du centre de $G/I$.
 Or $G/I\simeq\Gal(E(\alpha)/K)$ et $H/I\simeq\Gal(E(\alpha)/E)$.
 On peut donc supposer que $L:=E(\alpha)$ ce qu'on fera d\'esormais.
 On illustre la tour d'extensions que nous venons de d\'ecrire par le diagramme suivant:
 
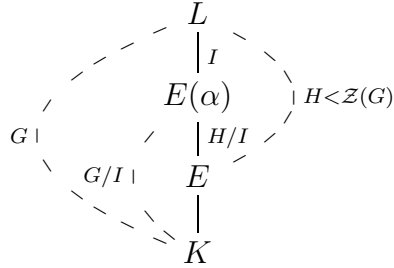
\[\xymatrix@-2ex{
  L \\
  E(\alpha) \ar@{-}[u]_I \\
  E \ar@{-}[u]_{H/I} \ar@/_3pc/@{--}[uu]_{H<\Zcal(G)} \\
  K \ar@{-}[u] \ar@/^2pc/@{--}[uu]^{G/I} \ar@/^5pc/@{--}[uuu]^{G}
 }\]
 Soit $\qcar$ un id\'eal premier de $\Oe$ au-dessus de $\pcar$.
 \begin{enumerate}
  \item Supposons que l'extension $E(\alpha)/E$ soit non ramifi\'ee en $\qcar$, alors, d'apr\`es la proposition \ref{propextBnonrami}, on a:
   \[\hcan_\psi(\alpha)\geq\frac{q^{-c_3\,\deg v\,d_1^2\,[K:k]}}{2\,[K:k]}\text{.}\]
  \item Supposons maintenant que l'extension $E(\alpha)/E$ soit ramifi\'ee en $\qcar$.
   Soit $\eta\in H_\qcar\setminus\{\id\}$
   (on rappelle que~$\Hqcar$ est non trivial d'apr\`es la proposition \ref{grpHpnontrivialB}).
   Comme $E=L(\alpha)$, on a $\eta(\alpha)-\alpha\neq0$.
   Alors d'apr\`es la condition (\ref{cond1B}), on a
   ${\eta(\alpha)-\alpha\notin\psi[\mathfrak{m}^\infty]}$.
   On peut donc appliquer la proposition \ref{propextBrami}, d'o\`u:
   \[\hcan_\psi(\alpha)\geq\frac{q^{-c_3\,\deg v\,d_1^2\,[K:k]}}{2\,[K:k]}\text{.}\]
 \end{enumerate}
 On obtient ainsi le r\'esultat annonc\'e.
\end{proof}

\subsection{Principe des tiroirs dans $(\OF/\mathfrak{m})^{\times}$}\label{principtiroirsBannquotient}
\leavevmode\par

Le but de ce paragraphe est d'\'etablir la proposition \ref{lemssgrpeltborne}
qui nous permettra au \S\ref{demothpsifinie} d'achever la preuve du th\'eor\`eme \ref{thlehmerBpsifini}.

Soient $F/k$ une extension finie de degr\'e $r$, $\OF$ la fermeture int\'egrale de~$A$ dans~$F$,
$\ncar$\label{notancar} un id\'eal non nul et principal de $\OF$
pour lequel on fixe un g\'en\'erateur $m\in\OF$. 
Soient $\Rcal_0$\label{notaRcal0} un syst\`eme de repr\'esentants de $\OF/\ncar$
et $e>0$\label{notae} un entier.
Pour tout $a\in\OF$, on note $\abar$\label{notaabar} l'image canonique de $a$ dans~$\OF/\ncar^e$.

Le lemme suivant permet de construire un syst\`eme explicite de repr\'esentants de $\OF/\ncar^e$
\`a partir du syst\`eme $\Rcal_0$ de repr\'esentants de $\OF/\ncar$.

\begin{lemm}\label{lembaseev}
 L'ensemble \label{notaRcal}$\displaystyle{\Rcal:=\!\left\{\sum_{i=0}^{e-1}\lambda_i\,m^i\,\big|\,\lambda_i\in\Rcal_0\!\right\}}$
 est un syst\`eme de repr\'esentants de $\OF/\ncar^e$.
\end{lemm}

\begin{proof}
 Remarquons tout d'abord que:
 \[\#\!\left(\mathcal{R}_0\right)^e=\#\!\left(\mathcal{O}_F/\mathfrak{n}\right)^e
  =q^{e\deg_{\mathcal{O}_F}\mathfrak{n}}=q^{\deg_{\mathcal{O}_F}\mathfrak{n}^e}=\#\!\left(\mathcal{O}_F/\mathfrak{n}^e\right)\text{.}\]
 Il suffit donc de montrer que les \'el\'ements de $\Rcal$ 
 sont deux-\`a-deux distincts modulo $\ncar^e$.
 Soient $\lambda_0,\dots,\lambda_{e-1},\mu_0,\dots,\mu_{e-1}\in\mathcal{R}_0$.
 Montrons donc que si:
 \[\sum_{i=0}^{e-1}\lambda_i\,m^i\equiv\sum_{i=0}^{e-1}\mu_i\,m^i\!\!\!\!\mod\ncar^e\text{,}\]
 alors, pour tout $i\in\{0,\dots,e-1\}$, on a $\lambda_i=\mu_i$. En effet, pour tout $j\in\{1,\dots,e-1\}$, on a:
 \[\sum_{i=0}^{e-1}\lambda_i\,m^i\equiv\sum_{i=0}^{e-1}\mu_i\,m^i\!\!\!\!\mod\mathfrak{n}^j
  \Longleftrightarrow\sum_{i=0}^{j-1}\lambda_i\,m^i\equiv\sum_{i=0}^{j-1}\mu_i\,m^i\!\!\!\!\mod\mathfrak{n}^j\text{.}\]
 On en d\'eduit successivement $\lambda_0=\mu_0, \lambda_1=\mu_1, \dots, \lambda_{e-1}=\mu_{e-1}$.
\end{proof}

Soit $N$\label{notaN} le nombre d'\'el\'ements de $\mathcal{O}_F$ de degr\'e (en $\mathcal{O}_F$) $<2$
(qui est fini d'apr\`es le lemme $5.6$ de \cite{HAYES}).
La proposition suivante, qui repose sur le principe des tiroirs,
est un des ingr\'edients clefs qui nous permettra d'\'eliminer l'hypoth\`ese technique (\ref{cond1B}) p. \pageref{cond1B}.

\begin{prop}\label{propssgrpR}
 Soient $B>1$\label{notaB} un entier
 et $H$\label{notaH2} un sous-groupe
 de $(\OF/\ncar^e)^{\times}$\label{notaOFsurncare} d'indice $<B$.
 On suppose que $e\geq\log_q(2\,N B)$.
 Alors, il existe deux \'el\'ements $a$ et $b$ de $\OF$ tels que $\abar\in H$, $\bbar\in H$ et:
 \[2<q^{\degOF(b-a)}<2\,N B\,q^{\maxN}\]
 o\`u $\maxN:=\max_{\lambda\in\Rcal_0}\{\degOF\lambda\}$\label{notamaxN}. 
\end{prop}

\begin{proof}
 D'apr\`es le lemme~\ref{lembaseev}, l'ensemble:
 \[\mathcal{R}:=\left\{\sum_{i=0}^{e-1}\lambda_i\,m^i\,\Big|\,\lambda_i\in\mathcal{R}_0\right\}\]
 est un syst\`eme de repr\'esentants des classes d'\'equivalence de $\mathcal{O}_F/\mathfrak{n}^e$. 

 \noindent Posons:
 \[\label{notac}
  c:=\left\lceil\frac{\log_q(2\,N B)}{\deg_{\mathcal{O}_F}\mathfrak{n}}\right\rceil\text{.}
 \]
 Soit
 \[\label{notaGamma}
  \Gamma:=\left\{\sum_{i=c}^{e-1}\lambda_i\,m^i\,\big|\,\lambda_i\in\mathcal{R}_0\right\}\text{.}
 \]
 Pour tout $\alpha\in \Gamma$, d\'efinissons
 \[\label{notaIalpha}
  I_\alpha:=\left\{\sum_{i=0}^{c-1}\lambda_i\,m^i+\alpha\,\big|\,\lambda_i\in\mathcal{R}_0\right\}\text{.}
 \]
 Ainsi, $\mathcal{R}$ est la r\'eunion disjointe de tous les $I_\alpha$ pour $\alpha$ parcourant $\Gamma$. 
 Soit 
 \[\label{notaLambda}
  \Lambda:=\left\{a\in\mathcal{O}_F\,\big|\,\abar\in H\right\}\text{.}
 \]
 Alors, $H$ est en bijection avec $\Lambda\cap\mathcal{R}$ qui est la r\'eunion disjointe des $\Lambda\cap I_\alpha$ pour $\alpha\in\Gamma$.
 Supposons que pour tout $\alpha\in \Gamma$, on ait $\#\!\left(\Lambda\cap I_\alpha\right)\leq N$.
 Alors:
 \[\#H=\#\!\left(\Lambda\cap\mathcal{R}\right)=\sum_{\alpha\in \Gamma}\#\!\left(\Lambda\cap I_\alpha\right)\leq N\,\#\Gamma\text{.}\]
 Or, $\#\Gamma=q^{(e-c)\deg_{\mathcal{O}_F}\mathfrak{n}}$ (car $e\geq\log_q(2\,N B)\geq c$), d'o\`u:
 \[\#H\leq N\,q^{(e-c)\deg_{\mathcal{O}_F}\mathfrak{n}}
  =N\,\frac{\#\!\left(\mathcal{O}_F/\mathfrak{n}^e\right)}{\#\!\left(\mathcal{O}_F/\mathfrak{n}^c\right)}\text{.}\]
 Ainsi, 
\begin{multline*}
 B>\left[(\mathcal{O}_F/\mathfrak{n}^e)^{\times}:H\right]
  =\frac{\#\!\left((\mathcal{O}_F/\mathfrak{n}^e)^{\times}\right)}{\#(H)} \\
  \geq\frac{\#\!\left(\mathcal{O}_F/\mathfrak{n}^c\right)}{N}
  \,\frac{\#\!\left((\mathcal{O}_F/\mathfrak{n}^e)^{\times}\right)}{\#\!\left(\mathcal{O}_F/\mathfrak{n}^e\right)}
  \geq\frac{1}{2}\,\frac{\#\!\left(\mathcal{O}_F/\mathfrak{n}^c\right)}{N}
\end{multline*}
 car, d'apr\`es la proposition $1.6$ de \cite{ROSEN}, on a:
 \[\#\!\left((\mathcal{O}_F/\mathfrak{n}^e)^{\times}\right)=\#\!\left(\OF/\ncar\right)^e\,\left(1-\frac{1}{\#\!\left(\OF/\ncar\right)}\right)
  \geq \frac{1}{2}\,\#\!\left(\mathcal{O}_F/\mathfrak{n}^e\right)\text{.}\]
 On obtient donc $2\,N\,B>q^{c\deg_{\mathcal{O}_F}\mathfrak{n}}$, ce qui contredit le choix de $c$.
 Il existe alors $\alpha\in \Gamma$ tel que $\#\!\left(\Lambda\cap I_\alpha\right)\geq N+1$.
 On peut donc trouver $a$ et~$b$ distincts dans $\Lambda\cap I_\alpha$
 tels que $\deg_{\mathcal{O}_F}(b-a)\geq2$.
 En effet, si $a$ est fix\'e dans $\Lambda\cap I_\alpha$, le cardinal de l'ensemble
 $\left\{b-a\,\big|\,b\in\Lambda\cap I_\alpha\right\}$ est sup\'erieur ou \'egal \`a $N+1$.
 Or, il n'y a que~$N$ \'el\'ements de $\OF$ de degr\'e $\leq1$.
 Donc il existe bien $b\in\Lambda\cap I_\alpha$ tel que $\deg_{\mathcal{O}_F}(b-a)\geq2$.
 Et comme $a,b\in I_\alpha$, on a \'egalement:
 \[\deg_{\mathcal{O}_F}(b-a)\leq (c-1)\deg_{\mathcal{O}_F}\mathfrak{n}+\maxN\text{.}\]
\end{proof}

On suppose de plus que l'extension $F/k$ est CM.
Soient $\varphi$ un $\OF$-module de Drinfeld de rang $r$, $\Ftilde:=F\cap\Fbar_q$ le corps des constantes de $F$,
$\cardFtilde:=\log_q\#\Ftilde$ (\ie $\Ftilde=\F_{q^\cardFtilde}$) et $\nbclass$ le nombre de classes de $\OF$.
Soient $\infty'$ l'unique place de $F$ au-dessus de $\infty$ et $d:=\deg\infty'$ le degr\'e de $\infty'$ sur~$\F_q$.
Soit~$\mcar$ un id\'eal premier non nul de $\OF$, alors $\mcar^\nbclass$ est un id\'eal principal
dont on note $m$ un g\'en\'erateur.
La proposition suivante est un raffinement du r\'esultat pr\'ec\'edent 
(on rappelle que la notation $\mu_{\varphi}$ est d\'efinie au \S\ref{subsectionThHayes}).

\begin{prop}\label{lemssgrpeltborne}
 Soient $B>1$ un entier
 et $H$ un sous-groupe de $(\OF/\mcar^{\nbclass e})^{\times}$ d'indice $<B$.
 On suppose que $e\geq\log_q(2\,B\,q^{\cardFtilde\,(d+1)})$. 
 Alors, il existe deux \'el\'ements $a$ et $b$ de $\OF$ tels que $\abar\in H$, $\bbar\in H$,
 $\degOF a=\degOF b$, $\mu_\varphi(a)=\mu_\varphi(b)$ et:
 \[2<q^{\deg_{\mathcal{O}_F}(b-a)}<2\,B\,q^{\cardFtilde\,(d+1)+\maxR}\]
 o\`u $\maxR:=\max_{\lambda\in\mathcal{R}_0}\{\deg_{\mathcal{O}_F}\lambda\}$. 
\end{prop}

\begin{proof}
 D'apr\`es la proposition $1.4.9$ de \cite{STICHT},
 le nombre $N$ d'\'el\'ements de $\OF$ de degr\'e $<2$ v\'erifie $N\leq q^{\cardFtilde\,(d+1)}$.
 
 
 D'apr\`es la proposition \ref{propssgrpR}, il existe $a'$ et $b'$ dans $\OF$ tels que
 $\overline{a'}\in H$, $\overline{b'}\in H$ et:
 \[2<q^{\deg_{\mathcal{O}_F}(b'-a')}<2\,B\,q^{\cardFtilde\,(d+1)+\maxR}\text{.}\]
 Soit $\lambda\in\mathcal{O}_F$ de degr\'e suffisamment grand, \ie:
 \[\degOF\lambda>\max\!\left\{\degOF a',\degOF b'\right\}\text{.}\]
 Posons $a=a'+\lambda\,m^e$ et $b=b'+\lambda\,m^e$, alors $\abar\in H$, $\bbar\in H$,
 \[\deg_{\mathcal{O}_F} b=\deg_{\mathcal{O}_F} a=\deg_{\mathcal{O}_F}\lambda+e\,\deg_{\mathcal{O}_F} m\]
 et
 \[\deg_{\mathcal{O}_F}(b-a)=\deg_{\mathcal{O}_F}(b'-a')<\deg_{\mathcal{O}_F} a\text{,}\]
 d'o\`u 
 \[\deg_\tau\varphi_a=r\,\deg_{\mathcal{O}_F} a>r\,\deg_{\mathcal{O}_F}(b-a)=\deg_\tau\varphi_{b-a}\text{.}\]
 Or $\varphi_{b-a}=\varphi_b-\varphi_a$, donc $\mu_\varphi(a)=\mu_\varphi(b)$.
\end{proof}

Le\ \,lemme\ \,suivant\ \,\'etablit\ \,une\ \,majoration\ \,de la constante $\maxR$ 
qui appara\^it dans la proposition \ref{lemssgrpeltborne} en fonction d'une base fix\'ee du $A$-module~$\OF$. 

\begin{lemm}\label{systrepborne}
 Soit $(e_1,\dots,e_r)$\label{notae1er} une base de $\OF$ en tant que $A$-mo\-dule.
 Il existe un syst\`eme de repr\'esentants $\mathcal{R}_0$ de $\OF/\mathfrak{m}^\nbclass$ tel que:
 \[
  \maxR:=\max_{\lambda\in\mathcal{R}_0}\{\deg_{\OF}\lambda\}\leq\maxB+\nbclass\,r\,\deg_{\OF}\mathfrak{m}
 \]
 o\`u $\maxB:=\max_{1\leq i\leq r}\{\deg_{\OF} e_i\}$\label{notamaxB}.
\end{lemm}


\begin{proof}
 Soit $m_0\in A$ le g\'en\'erateur unitaire de l'id\'eal premier de $A$ en-dessous de $\mathfrak{m}$
 (\emph{i.e.} $(m_0)=A\cap\mathfrak{m}$).
 Comme $m_0\in\mathfrak{m}$, on a ${m_0^\nbclass\in\mcar^\nbclass}$, donc $\OF/(m_0^\nbclass\OF)$ se surjecte sur $\OF/\mathfrak{m}^\nbclass$.
 Ainsi, tout syst\`eme de repr\'esentants de $\OF/(m_0^\nbclass\OF)$ contient un syst\`eme de repr\'esentants de $\OF/\mathfrak{m}^\nbclass$.
 Or, l'ensemble $\left\{a\in A\,\big|\,\deg_A a<\deg_A m_0^\nbclass\right\}$ est un syst\`eme de repr\'esentants de $A/(m_0^\nbclass)$.
 Donc l'ensemble~$\mathcal{R}_0$ des \'el\'ements $\lambda\in\OF$ tel que $\lambda=a_1\,e_1+\cdots+a_r\,e_r$
 avec les $a_i\in A$ de degr\'e $<\deg_A m_0^\nbclass$ forme un syst\`eme de repr\'esentants de $\OF/(m_0^\nbclass\OF)$.
 De plus, pour tout $a\in A$, on a $\degOF a=r\,\degA a$
 car:
 \[\OF/(a\OF)\simeq\left(A/(a)\right)^r\text{.}\]
 Ainsi, pour tout $\lambda\in\mathcal{R}_0$, on a:
 \begin{eqnarray*}
  \degOF\lambda=\degOF\!\left(a_1\,e_1+\cdots+a_r\,e_r\right)\!\!\!\!
  & \leq & \!\!\!\!\max_{1\leq i\leq r}\{\degOF a_i\,e_i\} \\
  & \leq & \!\!\!\!\max_{1\leq i\leq r}\{\degOF e_i\}+\max_{1\leq i\leq r}\{\degOF a_i\} \\
  & \leq & \!\!\!\!\maxB+r\,\max_{1\leq i\leq r}\{\degA a_i\} \\
  & \leq & \!\!\!\!\maxB+r\,\degA m_0^\nbclass \\
  & \leq & \!\!\!\!\maxB+\nbclass\,r\,\degA m_0\text{.}
 \end{eqnarray*}
 Or, $\deg_A m_0\leq\deg_{\OF}\mathfrak{m}$ (car $\OF/\mathfrak{m}$ est une extension finie de $A/(m_0)$), ce qui ach\`eve la preuve. 
\end{proof}

\subsection{D\'emonstration du th\'eor\`eme \ref{thlehmerBpsifini}}\label{demothpsifinie}
\leavevmode\par

Nous sommes maintenant en mesure de d\'emontrer le th\'eor\`eme \ref{thlehmerBpsifini}.



\begin{proof}[D\'emonstration du th\'eor\`eme \ref{thlehmerBpsifini}]
 Rappelons que comme $\psi$ est de signe normalis\'e son corps des coefficients est~$H_F^+$ (\emph{cf.} \S\ref{subsectionThHayes}).
 Donc ${K\supset\HFplus\supset F}$.
 Conform\'ement aux notations de la page \pageref{notademo}, notons $\pcar$ l'id\'eal de $\OK$ associ\'e \`a la place $v$
 et $\qcar$ un id\'eal premier de $\Oe$ au-dessus de $\pcar$.
 On pose ${\mcar:=\pcar\cap\OF}$, $\Ftilde:=F\cap\Fbar_q$ et
 $\cardFtilde:=\log_q\#\Ftilde$.
 Soient $\nbclass$ le nombre de classes de $\OF$ et $m$ un g\'en\'erateur de l'id\'eal principal~$\mcar^\nbclass$.
 Fixons $(e_1,\dots,e_r)$ une base de~$\OF$ en tant que $A$-module
 et $\mathcal{R}_0$ un syst\`eme de repr\'esentants de $\OF/\mcar^\nbclass$ tel que:
 \[\maxR:=\max_{\lambda\in\Rcal_0}\{\deg_{\OF}\lambda\}\leq \maxB+\nbclass\,r\,\deg_{\OF}\mcar\]
 o\`u $\maxB:=\max_{1\leq i\leq r}\{\deg_{\OF} e_i\}$ (un tel syst\`eme de repr\'esentants existe d'apr\`es le lemme \ref{systrepborne}).
 Soit $\alpha\in L$ non de torsion pour $\psi$.
 Soit $i$\label{notai} un entier suffisamment grand, \emph{i.e.} tel que:
 \begin{equation}\label{eqintersectionLptdetorsion}
  L\cap H_F^+\!\left(\psi\!\left[\mcar^{\infty}\right]\right)
  =L\cap H_F^+\!\left(\psi\!\left[\mcar^{\nbclass \nu i}\right]\right)
 \end{equation}
 et
 \[\nu\,i\geq\log_q\!\left(4\,d_1\,[K:H_F^+]\,q^{\cardFtilde\,(\dinfprime+1)}\right)\text{.}\]
 Alors, en posant:
 \[\label{notaH1}
  H_1:=\Gal\!\left(E\!\left(\psi\!\left[\mcar^{\nbclass \nu i}\right]\right)/E\right)\text{,}
 \]
 on a
 \[H_1\simeq
  \Gal\!\left(H_F^+\!\left(\psi\!\left[\mcar^{\nbclass \nu i}\right]\right)
  /E\cap H_F^+\!\left(\psi\!\left[\mcar^{\nbclass \nu i}\right]\right)\right)\text{.}\]
 On peut donc voir $H_1$
 comme un sous-groupe de $\Gal\!\left(H_F^+\!\left(\psi\!\left[\mcar^{\nbclass \nu i}\right]\right)/H_F^+\right)$.
 Posons:
 \[\label{notaG1}
  G_1:=\Gal\!\left(H_F^+\!\left(\psi\!\left[\mcar^{\nbclass \nu i}\right]\right)/H_F^+\right)\text{.}
 \]
 On illustre la tour d'extensions que nous venons de d\'ecrire par le diagramme suivant:
 \[\xymatrix@-3ex{
   & E\!\left(\psi\!\left[\mcar^{\nbclass \nu i}\right]\right) \ar@{-}[dl]^{H_1} \ar@{-}[dr]& \\
   E & & H_F^+\!\left(\psi\!\left[\mcar^{\nbclass \nu i}\right]\right) \\
   & E\cap H_F^+\!\left(\psi\!\left[\mcar^{\infty}\right]\right) \ar@{-}[ul] \ar@{-}[ur] & \\
   & \HFplus \ar@{-}[u] \ar@/_2pc/@{--}[uur]_{G_1} &
  }\]
 D'apr\`es la proposition \ref{propextpttorsiontotalrami},
 l'extension $H_F^+\!\left(\psi\!\left[\mcar^{\nbclass \nu i}\right]\right)/H_F^+$
 est totalement ramifi\'ee en tout premier de $H_F^+$ au-dessus de $\mcar$, ainsi:
 \[\left[E\cap H_F^+\!\left(\psi\!\left[\mcar^{\nbclass \nu i}\right]\right):H_F^+\right]\leq e_\qcar\!\left(E/H_F^+\right)\leq d_1\,[K:H_F^+]\text{.}\]
 D'autre part, d'apr\`es l'isomorphisme (\ref{eqisogrpGaloisquotientcroix}) p. \pageref{eqisogrpGaloisquotientcroix}, on a:
 \[G_1\simeq\left(\OF/\mcar^{\nbclass \nu i}\right)^\times\text{.}\]
 On peut donc voir $H_1$ comme un sous-groupe de $\left(\OF/\mcar^{\nbclass \nu i}\right)^\times$
 et lui appliquer la proposition \ref{lemssgrpeltborne} avec $B:=d_1\,[K:H_F^+]+1$\label{notaB2}.
 Ainsi, en notant pour tout $a\in\OF$, $\abar$ l'image canonique de $a$ dans $\OF/\mcar^{\nbclass \nu i}$,
 on obtient qu'il existe deux \'el\'ements $a$ et $b$ de $\OF$ tels que $\abar\in H_1$, $\bbar\in H_1$,
 $\degOF a=\degOF b$, $\mu_\psi(a)=\mu_\psi(b)$ et:
 \[2<q^{\degOF(b-a)}<2\,B\,q^{\cardFtilde\,(\dinfprime+1)+\maxR}\text{.}\]
 De plus, 
 d'apr\`es l'\'egalit\'e 
 (\ref{eqsymboleArtinidealpttors}) p. \pageref{eqsymboleArtinidealpttors},
 en notant $\sigma_a\in H_1$\label{notasigmaa}
 et $\sigma_b\in H_1$ les symboles d'Artin associ\'es, respectivement, aux id\'eaux $(a)$ et~$(b)$ de $\OF$,
 pour tout $\lambda\in\psi\!\left[\mcar^{\nbclass \nu i}\right]$, on a:\enlargethispage{1cm}
 \begin{equation}\label{relsigmaa}
  \mu_\psi(a)\,\sigma_a(\lambda)=\psi_a(\lambda)
 \end{equation}
 et
 \begin{equation}\label{relsigmab}
  \mu_\psi(b)\,\sigma_b(\lambda)=\psi_b(\lambda)\text{.}
 \end{equation}
 Soient $\tilde{\sigma}_a$ et $\tilde{\sigma}_b$ des prolongements respectifs de $\sigma_a$ et $\sigma_b$ \`a
 $L\!\left(\psi\!\left[\mathfrak{m}^{\nbclass \nu i}\right]\right)$.
 Ainsi $\tilde{\sigma}_a$ et $\tilde{\sigma}_b$ sont des \'el\'ements de
 $\Gal\!\left(L\!\left(\psi\!\left[\mathfrak{m}^{\nbclass \nu i}\right]\right)/E\right)$.
 On pose:
 \begin{equation}\label{eqdefbeta}
  \beta:=\left(\mu_\psi(b)\,\tilde{\sigma}_b-\mu_\psi(a)\,\tilde{\sigma}_a\right)^{q^\dinfprime-1}(\alpha)
   -\psi_{(b-a)^{q^\dinfprime-1}}(\alpha)\in L\text{.}
 \end{equation}
 On a:
 \begin{eqnarray*}
  \hcan_{\psi}(\beta) & = & \hcan_{\psi}\!\left(\left(\mu_\psi(b)\,\tilde{\sigma}_b
   -\mu_\psi(a)\,\tilde{\sigma}_a\right)^{q^\dinfprime-1}(\alpha)-\psi_{(b-a)^{q^\dinfprime-1}}(\alpha)\right) \\
  & \leq & \hcan_{\psi}\!\left(\left(\mu_\psi(b)\,\tilde{\sigma}_b-\mu_\psi(a)\,\tilde{\sigma}_a\right)^{q^\dinfprime-1}(\alpha)\right)
   +\hcan_{\psi}\!\left(\psi_{(b-a)^{q^\dinfprime-1}}(\alpha)\right) \text{.}
 \end{eqnarray*}
 Or,
 \begin{eqnarray*}
  \big(\mu_\psi(b)\,\tilde{\sigma}_b\!\!\!\!\!\!\!\!\!\!\!\!\!\!
   & - & \!\!\!\!\!\!\!\!\!\!\!\!\!\!\mu_\psi(a)\,\tilde{\sigma}_a\big)^{q^\dinfprime-1}(\alpha) \\
   & = & \sum_{j=0}^{q^\dinfprime-1}(-1)^j\,\binom{q^d-1}{j}\,\mu_\psi(a)^j\,\mu_\psi(b)^{q^\dinfprime-1-j}
    \,\tilde{\sigma}_b^{q^\dinfprime-1-j}\circ\tilde{\sigma}_a^j(\alpha) \\
   & \underset{\mu_\psi(a)=\mu_\psi(b)}{=} &
    \sum_{j=0}^{q^\dinfprime-1}(-1)^j\,\binom{q^d-1}{j}\,\mu_\psi(a)^{q^\dinfprime-1}\,\tilde{\sigma}_b^{q^\dinfprime-1-j}\circ\tilde{\sigma}_a^j(\alpha) \\
   & = & \sum_{j=0}^{q^\dinfprime-1}(-1)^j\,\binom{q^d-1}{j}\,\tilde{\sigma}_b^{q^\dinfprime-1-j}\circ\tilde{\sigma}_a^j(\alpha) 
 \end{eqnarray*}
 car $\mu_\psi(a)\in\F_{q^\dinfprime}^*$ d'apr\`es la relation (\ref{eqmupsidansFqdstar}) p. \pageref{eqmupsidansFqdstar}.
 Comme pour tout $x\in\kbar$ et tout $\zeta\in\F_q^*$
 on a\footnote{En effet, on a:\quad
  $\hcan_\psi(\zeta\,x)=\hcan_\psi(\psi_\zeta(x))=q^{\degOF\zeta}\,\hcan_\psi(x)=\hcan_\psi(x)$. \\
   On pourrait aussi remarquer que $\binom{q^d-1}{j}=(-1)^j$ dans~$\F_q$,
   car $\binom{q^d-1}{0}=1$ et:
   \centerline{$\binom{q^d-1}{j}+\binom{q^d-1}{j+1}=\binom{q^d}{j+1}\equiv0\mod p\text{.}$}
  }
 $\hcan_\psi(\zeta\,x)=\hcan_\psi(x)$, on en d\'eduit:
 \begin{equation}\label{eqhautcanpremiermembrebeta}
  \hcan_{\psi}\!\left(\left(\mu_\psi(b)\,\tilde{\sigma}_b-\mu_\psi(a)\,\tilde{\sigma}_a\right)^{q^\dinfprime-1}(\alpha)\right)
   \leq q^\dinfprime\,\hcan_{\psi}(\alpha)\text{.}
 \end{equation}
 Par ailleurs,
 \begin{equation}\label{eqhautcansecondmembrebeta}
  \hcan_{\psi}\!\left(\psi_{(b-a)^{q^\dinfprime-1}}(\alpha)\right)= q^{(q^\dinfprime-1)\deg_{\OF}(b-a)}\,\hcan_{\psi}(\alpha)\text{.}
 \end{equation}
 En regroupant les relations ci-dessus, on obtient:
 \[
  \hcan_\psi(\beta) 
   \leq\left(q^\dinfprime+q^{(q^\dinfprime-1) \deg_{\OF}(b-a)}\right)\hcan_{\psi}(\alpha)
   \leq2\,\left(q^{\deg_{\OF}(b-a)}\right)^{q^\dinfprime-1}\,\hcan_{\psi}(\alpha) \text{.}
 \]
 Or, par choix de $a$ et $b$, on a:
 \[2<q^{\deg_{\OF}(b-a)}<2\,B\,q^{\cardFtilde\,(\dinfprime+1)+\maxR}\text{,}\]
 avec $B=d_1\,[K:H_F^+]+1$, ainsi:
 \begin{eqnarray*} 
  \hcan_\psi(\beta) & \leq & 2\,\left(2\,(d_1\,[K:H_F^+]+1)\,q^{\cardFtilde\,(\dinfprime+1)+\maxR}\right)^{q^\dinfprime-1}\,\hcan_{\psi}(\alpha) \\
  & \leq & 2\,\left(4\,d_1\,[K:k]\,q^{\cardFtilde\,(\dinfprime+1)+\maxR}\right)^{q^\dinfprime-1}\,\hcan_{\psi}(\alpha)\text{.}
 \end{eqnarray*}
 Supposons dans un premier temps que les hypoth\`eses du th\'eor\`eme \ref{thBminorcondi}
 (avec $\alpha$ remplac\'e par $\beta$) soient v\'erifi\'ees.
 D'apr\`es ce th\'eor\`eme on a alors:
 \[\hcan_\psi(\beta)\geq\left(2\,[K:k]\,q^{3\,\nbclass\,\dinfprime\,(q^\dinfprime-1)\,\gamma(\psi)\,\degOK\pcar\,d_1^2\,[K:k]}\right)^{-1}\]
 et donc:
 \begin{multline*}
  \hcan_{\psi}(\alpha)\geq \\
   \left(4\,(4\,d_1)^{q^\dinfprime-1}\,[K:k]^{q^\dinfprime}\,q^{(q^\dinfprime-1)\left(3\,\nbclass\,\dinfprime
   \,\gamma(\psi)\,d_1^2\,\deg_{\OK}\pcar\,[K:k]+\cardFtilde\,(\dinfprime+1)+\maxR\right)}\right)^{-1}\text{.}
 \end{multline*}
 Or,
 \begin{multline*}
  \maxR\leq\maxB+\nbclass\,r\,\deg_{\OF}\mcar
   \leq2\,\max\{1,\maxB\}\,\nbclass\,r\,\deg_{\OF}\mcar \\
   \leq2\,\max\{1,\maxB\}\,\nbclass\,r\,\deg_{\OK}\pcar\text{.}
 \end{multline*}
 D'o\`u
 \begin{eqnarray*}
  3\,\nbclass\,\dinfprime\!\!\!\! & \gamma(\psi) & \!\!\!\!d_1^2\,\deg_{\OK}\pcar\,[K:k]+\cardFtilde\,(\dinfprime+1)+\maxR \\
  & \leq & \!\!\!\!\!\!\!\!3\,\nbclass\,\dinfprime\,\gamma(\psi)\,d_1^2\,\deg_{\OK}\pcar\,[K:k]
   +4\,\max\{1,\maxB\}\,\nbclass\,r\,\cardFtilde\,\deg_{\OK}\pcar \\
  & \leq & \!\!\!\!\!\!\!\!\nbclass\,\deg_{\OK}\pcar\,\left(3\,\dinfprime\,\gamma(\psi)\,d_1^2\,[K:k]+4\,r\,\cardFtilde\,\max\{1,\maxB\}\right) \\
  & \leq & \!\!\!\!\!\!\!\!7\,\nbclass\,r\,\cardFtilde\,\max\{1,\maxB\}\,\gamma(\psi)\,d_1^2\,\deg_{\OK}\pcar\,[K:k] \text{.}
 \end{eqnarray*}
 Ainsi,
 \[
  \hcan_\psi(\alpha)\geq
   \left(4\,(4\,d_1)^{q^\dinfprime-1}\,[K:k]^{q^\dinfprime}\,q^{7\,\nbclass\,r\,\cardFtilde\,(q^\dinfprime-1)
   \,\max\{1,\maxB\}\,\gamma(\psi)\,d_1^2\,\deg_{\OK}\pcar\,[K:k]}\right)^{-1}\text{,}
 \]
 ce qui nous donne le r\'esultat \'enonc\'e en posant:
 \[\label{notac1}
  c_1:=7\,\nbclass\,r\,\cardFtilde\,(q^\dinfprime-1)\,\max\{1,\maxB\}\,\gamma(\psi)\text{,}
 \]
 et $\displaystyle{c_2:=4^{q^\dinfprime}}$.\label{notac2}
 Il est clair que les constantes $c_1$ et $c_2$ sont positives et ne d\'ependent que de $\psi$.
 Plus pr\'ecis\'ement, elles ne d\'ependent que de $\nbclass$, $\maxB$, $r$, $\dinfprime$ et $\gamma(\psi)$.
 Or~$\gamma(\psi)$ ne d\'epend que de~$\psi$
 et $\nbclass$, $\maxB$,~$r$ et $\dinfprime$ ne d\'ependent que de l'anneau $\OF$,
 donc de $\psi$ car~$\OF$ est l'anneau des endomorphismes de $\psi$.

 Il ne reste donc plus qu'\`a v\'erifier que $\beta$ n'est pas un point de torsion pour $\psi$
 et que $\beta$ v\'erifie la condition (\ref{cond1B}) du th\'eor\`eme \ref{thBminorcondi}:
 pour tout $\eta\in\Gal\!\left(L/E\right)$,
 on a $\eta(\beta)-\beta\notin\psi\!\left[\mathfrak{m}^{\infty}\right]\setminus\{0\}$.
 Supposons par l'absurde que~$\beta$ soit un point de torsion pour $\psi$.
 Dans ce cas, on a d'apr\`es la d\'efinition (\ref{eqdefbeta}) p.~\pageref{eqdefbeta} de $\beta$
 et les relations (\ref{eqhautcanpremiermembrebeta}) et (\ref{eqhautcansecondmembrebeta}) p.~\pageref{eqhautcansecondmembrebeta}:
 \begin{multline*}
  q^{(q^\dinfprime-1)\deg_{\OF}(b-a)}\,\hcan_{\psi}(\alpha)=\hcan_{\psi}\!\left(\psi_{(b-a)^{q^\dinfprime-1}}(\alpha)\right) \\
   =\hcan_{\psi}\!\left(\left(\mu_\psi(b)\,\tilde{\sigma}_b-\mu_\psi(a)\,\tilde{\sigma}_a\right)^{q^\dinfprime-1}(\alpha)\right)
   \leq q^\dinfprime\,\hcan_{\psi}(\alpha)\text{.}
 \end{multline*}
 Or, $q^{\deg_{\OF}(b-a)}>2$ et $2^{q^\dinfprime-1}\geq q^\dinfprime$ donc $\hcan_\psi(\alpha)=0$,
 ce qui contredit le fait que $\alpha$ ne soit pas un point de torsion pour~$\psi$.
 Montrons maintenant que $\beta$ v\'erifie la condition~(\ref{cond1B}) du th\'eor\`eme \ref{thBminorcondi}.
 Pour cela supposons par l'absurde qu'il existe
 ${\eta\in\Gal\!\left(L/E\right)}$\label{notaeta}
 tel que ${\theta:=\eta(\beta)-\beta\in L\cap\psi\!\left[\mathfrak{m}^{\infty}\right]\setminus\{0\}}$\label{notatheta}.
 Alors, en posant $\delta:=\eta(\alpha)-\alpha\in L$\label{notadelta},
 on a d'apr\`es la d\'efinition (\ref{eqdefbeta}) p. \pageref{eqdefbeta} de $\beta$
 et en utilisant le fait que
 $\Gal\!\left(L/E\right)$
 est ab\'elien:
 \begin{eqnarray*}
  \theta=\eta(\beta)-\beta & = &
   \eta\!\left(\left(\mu_\psi(b)\,\tilde{\sigma}_b-\mu_\psi(a)\,\tilde{\sigma}_a\right)^{q^\dinfprime-1}(\alpha)
    -\psi_{(b-a)^{q^\dinfprime-1}}(\alpha)\right) \\
   & & \quad-\left(\mu_\psi(b)\,\tilde{\sigma}_b-\mu_\psi(a)\,\tilde{\sigma}_a\right)^{q^\dinfprime-1}(\alpha)
    -\psi_{(b-a)^{q^\dinfprime-1}}(\alpha) \\
   & = & \left(\mu_\psi(b)\,\tilde{\sigma}_b-\mu_\psi(a)\,\tilde{\sigma}_a\right)^{q^\dinfprime-1}(\eta(\alpha))
    -\psi_{(b-a)^{q^\dinfprime-1}}(\eta(\alpha)) \\
   & & \quad-\left(\mu_\psi(b)\,\tilde{\sigma}_b-\mu_\psi(a)\,\tilde{\sigma}_a\right)^{q^\dinfprime-1}(\alpha)
    -\psi_{(b-a)^{q^\dinfprime-1}}(\alpha) \\
   & = & \left(\mu_\psi(b)\,\tilde{\sigma}_b-\mu_\psi(a)\,\tilde{\sigma}_a\right)^{q^\dinfprime-1}(\eta(\alpha)-\alpha) \\
   & & \quad-\psi_{(b-a)^{q^\dinfprime-1}}(\eta(\alpha)-\alpha) \\
   & = & \left(\mu_\psi(b)\,\tilde{\sigma}_b-\mu_\psi(a)\,\tilde{\sigma}_a\right)^{q^\dinfprime-1}(\delta)
    -\psi_{(b-a)^{q^\dinfprime-1}}(\delta)\text{.}
 \end{eqnarray*}
 Comme $\theta$ est un point de torsion pour $\psi$, on a d'apr\`es
 les relations (\ref{eqhautcanpremiermembrebeta}) et (\ref{eqhautcansecondmembrebeta})
 p. \pageref{eqhautcansecondmembrebeta} (avec $\alpha$ remplac\'e par $\delta$):
 \begin{eqnarray*}
  q^{(q^\dinfprime-1)\deg_{\OF}(b-a)}\,\hcan_{\psi}(\delta) & = & \hcan_{\psi}\!\left(\psi_{(b-a)^{q^\dinfprime-1}}(\delta)\right) \\
   & = & \hcan_{\psi}\!\left(\left(\mu_\psi(b)\,\tilde{\sigma}_b-\mu_\psi(a)\,\tilde{\sigma}_a\right)^{q^\dinfprime-1}(\delta)\right) \\
   & \leq & q^\dinfprime\,\hcan_{\psi}(\delta)\text{.}
 \end{eqnarray*}
 Donc $\delta$ est un point de torsion pour $\psi$ car $q^{\deg_{\OF}(b-a)}>2$, et ${2^{q^\dinfprime-1}\geq q^\dinfprime}$.
 D'apr\`es le lemme \ref{lemsompttorsion}, il existe $\delta_1\in L$\label{notadelta1}
 et $\delta_2\in L$\label{notadelta2} deux points de torsion de~$\psi$ tels que $\delta=\delta_1+\delta_2$
 avec $\delta_1$ d'ordre une puissance de $\mathfrak{m}$ et~$\delta_2$ d'ordre premier \`a~$\mathfrak{m}$.
 Remarquons que $\delta_1\in\psi\!\left[\mcar^{\nbclass \nu i}\right]$ par choix de $i$ (\cf relation~(\ref{eqintersectionLptdetorsion})).
 Ainsi, en utilisant les relations (\ref{relsigmaa}) et (\ref{relsigmab})
 p. \pageref{relsigmaa}, on obtient:
 \[\left(\mu_\psi(b)\,\tilde{\sigma}_b-\mu_\psi(a)\,\tilde{\sigma}_a\right)^{q^\dinfprime-1}(\delta_1)
   -\psi_{(b-a)^{q^\dinfprime-1}}(\delta_1)=0\text{.}\]
 Donc
 \[\theta=\left(\mu_\psi(b)\,\tilde{\sigma}_b-\mu_\psi(a)\,\tilde{\sigma}_a\right)^{q^\dinfprime-1}(\delta_2)
  -\psi_{(b-a)^{q^\dinfprime-1}}(\delta_2)\text{.}\]
 Puisque $\delta_2$ est d'ordre premier \`a $\mcar$, on en d\'eduit que $\theta$ est \'egalement d'ordre premier \`a $\mcar$.
 Il existe donc un id\'eal non nul $\bcar$ de $\OF$ premier \`a $\mcar$ tel que $\theta\in\psi[\bcar]$.
 D'autre part, $\theta$ est d'ordre une puissance de $\mcar$ par hypoth\`ese.
 Il existe donc un entier $j>0$ tel que $\theta\in\psi[\mcar^j]$.
 Mais les id\'eaux $\bcar$ et $\mcar$ \'etant premiers entre eux, il en est de m\^eme de $\bcar$ et $\mcar^j$.
 Ainsi, il existe $u\in\bcar$ et $v\in\mcar^j$ tels que $u+v=1$.
 Or, $\psi_u(\theta)=0$ car $\theta\in\psi[\bcar]$ et $\psi_v(\theta)=0$ car $\theta\in\psi[\mcar^j]$.
 D'o\`u
 \[\theta=\psi_1(\theta)=\psi_u(\theta)+\psi_v(\theta)=0\text{.}\]
 Ce qui contredit l'hypoth\`ese suivant laquelle $\theta\neq0$.
 Donc $\beta$ v\'erifie bien la condition~(\ref{cond1B}) du th\'eor\`eme~\ref{thBminorcondi},
 ce qui ach\`eve la preuve du th\'eor\`eme.
\end{proof}

\section{D\'emonstration du th\'eor\`eme \ref{thlehmerB}}\label{demothlehmerB}

Dans ce paragraphe nous montrons le r\'esultat principal. 
%

\begin{proof}[D\'emonstration du th\'eor\`eme \ref{thlehmerB}]
 Notons $r$ le rang de $\phi$, $R$ l'image de $\End(\phi)$ dans $\kbar$ via l'isomorphisme (\ref{eqplongementEndphikbar}) p. \pageref{eqplongementEndphikbar},
 $F$ le corps des fractions de $R$. 
 On a vu au \S\ref{subsectionEndoModDrinfeld} que $F$ est une extension CM de $k$ de degr\'e $r$
 et que $R$ est un ordre de $F$.
 On a \'egalement vu qu'on pouvait naturelement munir~$\phi$ d'une structure de $R$-module de rang $1$.
 
 Soient $\phi'$ le $R$-module de Drinfeld \'etendu de $\phi$ (\cf \'egalit\'e (\ref{eqextscalairesModDrinfeld}) p. \pageref{eqextscalairesModDrinfeld})
 et $\Ccar$ le conducteur de~$R$. D'apr\`es le \S\ref{subsectionThHayes}, le $R$-module de Drinfeld $\varphi'=\Ccar*\phi'$
 est isog\`ene \`a $\phi'$ via $\phi'_\Ccar$ et tel que $\End(\varphi')=\OF$.
 Le $R$-module $\varphi'$ peut donc s'\'etendre en un $\OF$-module $\varphi$ de rang $1$.
 Soient $\infty'$ l'unique place de $F$ au-dessus de la place $\infty$,
 $d$ le degr\'e sur $\F_q$ de la place $\infty'$, $\pi\in F$ une uniformisante de $\infty'$
 et $z\in\kbar$ une racine du polyn\^ome $X^{q^d-1}-\mu_\varphi(\pi^{-1})$.
 Alors, d'apr\`es la proposition \ref{propracinepolysgnnormalise},
 le $\OF$-module de Drinfeld $\psi:=z\,\varphi\,z^{-1}$ est de signe normalis\'e.
 De plus, comme $\psi=z\,\varphi\,z^{-1}$ et $\varphi':=\Ccar*\phi'$,
 le corps $F_{\phi}(z)$\label{notaFphiz} est un corps de d\'efinition de $\psi$.
 D'apr\`es la proposition \ref{propEgalHautCanparExtScal}, pour tout $\alpha\in\kbar$, on a:
 \[\hcan_\phi(\alpha)=\hcan_{\phi'}(\alpha)\quad\text{et}\quad\hcan_{\varphi'}(\alpha)=\hcan_\varphi(\alpha)\text{.}\]
 D'autre part, d'apr\`es la proposition \ref{prophcanetisog}, pour tout $\alpha\in\kbar$, on a:
 \[q^{\deg_R\Ccar}\,\hcan_{\phi'}(\alpha)=\hcan_{\varphi'}\!\left(\phi'_\Ccar(\alpha)\right)\]
 et
 \[\hcan_\varphi(\alpha)=\hcan_\psi(z\,\alpha)\]
 car $\deg_\tau\phi'_\Ccar=\deg_R\Ccar$ et $\deg_\tau z=0$ (puisque $z\in\kbar^*$).
 Ainsi
 \begin{equation}\label{eqRelationEntreHautCan}
  q^{\deg_R\Ccar}\,\hcan_\phi(\alpha)=\hcan_\psi\!\left(z\,\phi'_\Ccar(\alpha)\right)\text{.}
 \end{equation}
 
 Nous allons appliquer le th\'eor\`eme \ref{thlehmerBpsifini} \`a $\psi$ et au point $\beta=z\,\phi'_\Ccar(\alpha)$,
 ce qui nous donnera une minoration de $\hcan_\phi(\alpha)$.

 Remarquons tout d'abord qu'on peut supposer que $K$ contient $F_{\phi}(z)$.
 En effet, en posant ${K_1\!\!:=\!K\,F_{\phi}(z)}$\label{notaK1} et $L_1:=L\,F_{\phi}(z)$\label{notaL1},
 l'extension $L_1/K_1$ est galoisienne.
 Soient $G_1:=\Gal(L_1/K_1)$\label{notaG12} et $E_1:=L_1^{\Zcal(G_1)}$\label{notaE1}.
 On v\'erifie ais\'ement que si $E_0:=E\,F_{\phi}(z)$\label{notaE0}, alors $\Gal(L_1/E_0)$ est un sous-groupe de~$\Zcal(G_1)$,
 donc $E_1\subset E_0$.
 On illustre la tour d'extensions que nous venons de d\'ecrire par le diagramme suivant:
 
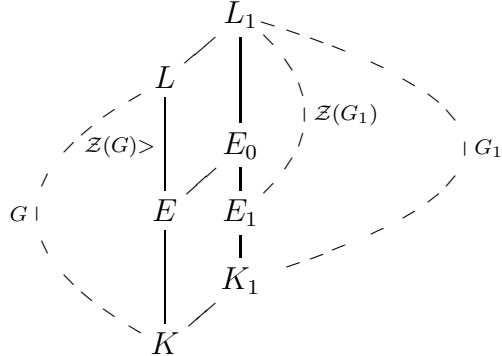
\[\xymatrix@-3ex{
   & L_1 \\
   L \ar@{-}[ur] & \\
   & E_0 \ar@{-}[uu] \\
   E \ar@{-}[uu]^{\Zcal(G)>} \ar@{-}[ur] & E_1 \ar@{-}[u] \ar@/_2pc/@{--}[uuu]_{\Zcal(G_1)} \\
   & K_1 \ar@{-}[u] \ar@/_7pc/@{--}[uuuu]_{G_1} \\
   K \ar@{-}[uu] \ar@{-}[ur] \ar@/^4pc/@{--}[uuuu]^{G} &
  }\]
 De plus, si $v_1$\label{notav1} est une place de $K_1$ au-dessus de $v$,
 alors pour toute place~$w_1$\label{notaw1} de $E_1$ au-dessus de $v_1$,
 on a $\left[(E_1)_{w_1}:(K_1)_{v_1}\right]\leq d_0$. En effet:
 \[\left[(E_1)_{w_1}:(K_1)_{v_1}\right]\leq\left[(E_0)_{w_0}:(K_1)_{v_1}\right]
  \leq\left[E_w:K_v\right]\leq d_0\text{,}\]
 o\`u $w_0$\label{notaw0} est une place de $E_0$ au-dessus de $w_1$
 et $w$\label{notaw} est la restriction de~$w_0$ \`a $E$.
 Ainsi, s'il existe une constante $c'_0\!>\!0$ telle que pour tout ${\alpha\in L_1\!\setminus\!\phi(L_1)_\tors}$, on~ait:
 \[\hcan_\phi(\alpha)\geq q^{-c'_0\,\deg v_1\,d_0^2\,[K_1:k]}\text{,}\]
 alors, en posant $\displaystyle{c_0:=c'_0\,[F_\phi(z):k]^2}$,
 pour tout $\alpha\in L\setminus\phi(L)_\tors$, on a:
 \[\hcan_\phi(\alpha)\geq q^{-c_0\,\deg v\,d_0^2\,[K:k]}\]
 car
 \[\deg v_1\leq[F_\phi(z):k]\,\deg v\]
 et
 \[[K_1:K]\leq[F_\phi(z):k]\,[K:k]\text{.}\]
 Donc, quitte \`a remplacer  $(K,L)$ par $(K_1,L_1)$, on peut supposer que $F_\phi(z)\subset K$,
 ce qu'on fera d\'esormais.
 
 Soit $L'/K$\label{notaLprime} une extension galoisienne finie telle que $L'\subset L$
 et notons ${G':=\Gal(L'/K)}$\label{notaGprime} son groupe de Galois.
 Soit $E':=L'\cap E$\label{notaEprime}.
 Alors, l'extension $L'/E'$ est galoisienne et $\Gal(L'/E')$ est un sous-groupe de $\Zcal(G')$.
 De plus, pour toute place $w|v$ de $E'$, on a ${[E'_w:K_v]\leq d_0}$.
 Soit $\beta\in L'$ qui n'est pas un point de torsion de $\psi$.
 Alors, d'apr\`es le th\'eor\`eme \ref{thlehmerBpsifini}, il existe deux constantes $c_1>0$ et $c_2\geq1$
 qui ne d\'ependent que de $\psi$ (et donc que de $\phi$) telles que:
 \[\hcan_{\psi}(\beta)\geq\frac{q^{-c_1\,\deg v\,d_0^2\,[K:k]}}{c_2\,d_0^{q^d-1}\,[K:k]^{q^d}}\text{.}\]
 Or, d'apr\`es la preuve du th\'eor\`eme \ref{thlehmerBpsifini}, on a:
 \[c_1:=7\,\nbclass\,r\,\cardFtilde\,(q^\dinfprime-1)\,\max\{1,\maxB\}\,\gamma(\psi)\]
 et $\displaystyle{c_2:=4^{q^\dinfprime}}$.
 En majorant grossi\`erement, on obtient:
 \[
  c_2\,d_0^{q^d-1}\,[K:k]^{q^d}\leq q^{2\,q^d}\,q^{(q^d-1)\,d_0}\,q^{q^d\,[K:k]}
   \leq q^{4\,q^d\,d_0\,[K:k]}\text{.}
 \]
 On en d\'eduit:
 \[\hcan_{\psi}(\beta)\geq q^{-11\,\nbclass\,r\,\cardFtilde\,q^d\,\max\{1,\maxB\}\,\gamma(\psi)\,\deg v\,d_0^2\,[K:k]}\text{.}\]
 Soit $\alpha\in L\setminus\phi(L)_\tors$.
 En substituant $\beta$ par $z\,\phi'_\Ccar(\alpha)$ dans l'in\'egalit\'e pr\'ec\'edente
 et en utilisant l'\'egalit\'e (\ref{eqRelationEntreHautCan}) p. \pageref{eqRelationEntreHautCan}, on obtient:
 \[
  \hcan_{\phi}(\alpha)=\frac{\hcan_\psi\!\left(z\,\phi'_\Ccar(\alpha)\right)}{q^{\degR\Ccar}}
   \geq q^{-11\,\nbclass\,r\,\cardFtilde\,q^d\,\max\{1,\maxB\}\,\gamma(\psi)\,\max\{1,\degR\Ccar\}\,\deg v\,d_0^2\,[K:k]}\text{.}
 \]
 D'o\`u le r\'esultat souhait\'e en posant:
 \[\label{notacprime0}
  c'_0:=11\,\nbclass\,r\,\cardFtilde\,q^d\,\max\{1,\maxB\}\,\gamma(\psi)\,\max\{1,\degR\Ccar\}\text{.}
 \]
\end{proof}

\begin{rema}
 Il est possible d'expliciter la constante $c_0$ du th\'eor\`eme pr\'ec\'edent.
 En effet, en reprenant la preuve du th\'eor\`eme \ref{thlehmerB}, on obtient:
 \[\label{notac0}
  c_0=11\,\nbclass\,r\,\cardFtilde\,q^d\,\max\{1,\maxB\}\,\gamma(\psi)\,\max\{1,\degR\Ccar\}\,[F_\phi(z):k]^2\text{.}
 \]
 En particulier, dans le cas du module de Carlitz (\ie le $A$-module de Drinfeld $C$ de rang $1$ d\'efini sur $k$ par $C_T=T\,\tau^0+\tau$),
 on obtient que pour toute extension finie $K/k$ et pour tout $\alpha\in\Kab\setminus C(\Kab)_\tors$, on a:
 \[\hcan_C(\alpha)\geq q^{-11\,q\,\gamma(C)\,[K:k]^2}\text{.}\]
\end{rema}

\begin{rema}
 Comme corollaire \`a leur th\'eor\`eme (\cf cor. $1.3$ de \cite{ADZ}), F. Amoroso, S. David et U. Zannier
 montrent qu'une extension galoisienne infinie de groupe de Galois d'exposant fini a la propri\'et\'e~(B).
 En effet, dans sa th\`ese (\cf \cite{SaraTh}), S. Checcoli montre que si $K$ est un corps de nombres
 et si $L/K$ est une extension galoisienne infinie de groupe de Galois $G$ d'exposant fini,
 alors les degr\'es locaux sur $L$ sont uniform\'ement born\'es en toutes les places de $K$.
 Ensuite S.~Checcoli et P.~D\`ebes (\cf \cite{CHECCODEB}) ont g\'en\'eralis\'e ce r\'esultat \`a certaines classes de corps de fonctions.
 Une question naturelle est donc de savoir si ce r\'esultat est encore vrai dans le cadre des corps de fonctions en caract\'eristique positive.
 Cette question fera l'objet d'un prochain article.
\end{rema}

%

\backmatter

\section*{Remerciements}

Je souhaite remercier Francesco Amoroso et Vincent Bosser mes directeurs de th\`ese ainsi que Bruno Angl\`es pour toutes les discussions que nous avons eus \`a propos de ce travail.

\bibliographystyle{smfalpha}
\bibliography{mabiblio}

\end{document}